\documentclass[12pt]{article}

\usepackage{stmaryrd}
\usepackage{mathrsfs}
\usepackage{amsmath,amsthm,amssymb}
\usepackage[colorlinks,
            linkcolor=red,
            anchorcolor=blue,
            citecolor=magenta
            ]{hyperref}
\usepackage{pdfsync}
\usepackage[numbers,sort&compress]{natbib}
\allowdisplaybreaks

\theoremstyle{definition}

\theoremstyle{plain}
\newtheorem{theorem}{Theorem}[section]
\newtheorem{definition}{Definition}[section]
\newtheorem{remark}{Remark}[section]
\newtheorem{lemma}{Lemma}[section]

\newtheorem{proposition}{Proposition}[section]

\newtheorem{corollary}{Corollary}[section]

\numberwithin{equation}{section}

\newcommand{\vs}{\vspace}

\begin{document}

\title{Global well-posedness, scattering and blow-up for
the energy-critical, Schr\"odinger
equation with indefinite potential in the radial case\footnote{  This work was partially supported by NNSFC (No. 12171493).}}

\author{ Jun Wang$^{a}$, Zhaoyang Yin$^{a, b}$\footnote {Corresponding author. wangj937@mail2.sysu.edu.cn (J. Wang), mcsyzy@mail.sysu.edu.cn (Z. Yin)
} \\
%EndAName
{\small $^{a}$Department of Mathematics, Sun Yat-sen University, Guangzhou, 510275, China } \\
{\small $^{b}$School of Science, Shenzhen Campus of Sun Yat-sen University, Shenzhen, 518107, China } \\
}

	\date{}

	\maketitle

\date{}

 \maketitle \vs{-.7cm}

  \begin{abstract}
In this paper, we
study the well-posedness theory and the scattering asymptotics for the energy-critical, Schr\"odinger equation with indefinite potential
\begin{equation*}
  \left\{\begin{array}{l}
i \partial_t u+\Delta u-V(x)u +|u|^{\frac{4}{N-2}}u=0,\ (x, t) \in \mathbb{R}^N \times \mathbb{R}, \\
\left.u\right|_{t=0}=u_0 \in  H ^1(\mathbb{R}^N),
\end{array}\right.
\end{equation*}
where $V(x):\mathbb{R}^N\rightarrow \mathbb{R}$ is indefinite and satisfies appropriate conditions. Using contraction mapping method and concentration compactness argument, we obtain
the well-posedness theory in proper function spaces and scattering asymptotics. Moreover, we get a positive ground state solution which is radially symmetric by using variational methods. This paper extends the results of \cite{KCEMF2006}(Invent. Math) to the potential equation and develops the recent conclusions.
\end{abstract}

{\footnotesize {\bf   Keywords:} Schr\"odinger equation, Indefinite potential, Well-posedness, Scattering

}

\section{ Introduction and main results}

This paper studies the well-posedness theory and the scattering asymptotics for the energy-critical, focusing, Schr\"odinger equation  with potential
\begin{equation}\label{eq1.1}
  \left\{\begin{array}{l}
i \partial_t u+\Delta u-V(x)u +|u|^{\frac{4}{N-2}}u=0,\ (x, t) \in \mathbb{R}^N \times \mathbb{R}, \\
\left.u\right|_{t=0}=u_0 \in  H ^1(\mathbb{R}^N),
\end{array}\right.
\end{equation}
where $V(x):\mathbb{R}^N\rightarrow \mathbb{R}$ is indefinite and satisfies appropriate conditions.

When $V(x)=0$, Kenig and Merle in \cite{KCEMF2006} study the $\dot{H}^1$ critical non-linear Schr\"odinger equation
$$
\left\{\begin{array}{l}
i \partial_t u+\Delta u \pm|u|^{\frac{4}{N-2}} u=0, \quad(x, t) \in \mathbb{R}^N \times \mathbb{R}, \\
\left.u\right|_{t=0}=u_0 \in \dot{H}^1(\mathbb{R}^N) .
\end{array}\right.
$$
Here the $-$ sign corresponds to the defocusing problem, while the $+$ sign corresponds to the focusing problem. They obtained  the global well-posedness, scattering and blow-up results in the radial case and $3\leq N\leq5$. Recently, Oh and Wang in \cite{TOYW2020} considered global well-posedness for one-dimensional cubic nonlinear Schr\"odinger equation by introducing a new function space. For $N=2$, \cite{JBAB2014} established global well-posedness results in the defocusing case, posed on the two-dimensional unit ball. For high dimensions case, Tao et.al in \cite{TTMV2007} established global well-posedness and scattering for solutions to the defocusing mass-critical nonlinear Schr\"odinger equation. For more well-posedness results, please refer to \cite{{XYHY2024},{DHNI2024},{HBGP2022},{JB1999},{JBM1999},{TCS2003},{TC1990}}.

When $V(x)\neq0$, two situations arise. If $V(x)$ is nonnegative potential(inverse square potential), Lu et al. in  \cite{JLCM2018} studied the scattering/blowup dichotomy below the ground state in the focusing case and proved scattering in $H^1$ for arbitrary data in the defocusing case. For more nonnegative potential results, please refer to \cite{{YH2016},{MHMI2020},{JZJZ2014},{RKCM2017}}. If $V(x)$ haves a negative part, it is not easy to handle. More precisely, it is very difficult for us to obtain the existence of the ground state solution of \eqref{eq1.1} and some necessary variational estimates. It is worth mentioning that the above mentioned papers only consider the subcritical energy case, so it is natural to ask whether similar conclusions hold for the critical energy case? The answer is affirmative, moreover, we have also proved that the solution blow up in infinite time as long as the initial energy is less than $0$.

In order to analyze the scattering asymptotics of \eqref{eq1.1}, we need to prove the existence of standing wave solutions to \eqref{eq1.1}, that is, $\Psi(x, t)=$ $e^{i \lambda t} u(x), \lambda \in \mathbb{R}$ and $u: \mathbb{R}^N \rightarrow \mathbb{R}$, so we get the equation
 \begin{equation}\label{eq1.99}
   -\Delta u+V(x)u=|u|^{\frac{4}{N-2}}u ,\ x\in\mathbb{R}^N,
 \end{equation}
where $V(x):\mathbb{R}^N\rightarrow \mathbb{R}$ is indefinite. This proof is based on the variational method, please refer to section 6 for details.

If the solution $u$ of \eqref{eq1.1} has sufficient decay at infinity and smoothness, it satisfies the conservation of mass
\begin{equation}\label{eq1.2}
  M(u(t)):=\|u(t)\|_{L^2}=\|u_0\|_{L^2}
\end{equation}
and the conservation of energy
\begin{equation}\label{eq1.3}
  E(u(t))=E(u_0),
\end{equation}
where $E(u(t))$ is defined by
\begin{equation*}
  E(u(t))=  \frac{1}{2} \int_{\mathbb{R}^N}(\left|\nabla u(x,t)\right|^2+V(x)|u(x,t)|^2) d x-   \frac{N-2}{2N}\int_{\mathbb{R}^N}|u(t)|^{\frac{2N}{N-2}} d x
\end{equation*}
and the energy space is $H^1$.

In this paper, the potential $V: \mathbb{R}^N \rightarrow \mathbb{R}$ is assumed to satisfy the following assumptions
\begin{equation}\label{eq1.4}
V \in \mathcal{K} \cap L^{\frac{N}{2}}
\end{equation}
and
\begin{equation}\label{eq1.5}
\left\|V_{-}\right\|_{\mathcal{K}}<N(N-2)\alpha(N),
\end{equation}
where $\alpha(N)$ denotes the volume of the unit ball in $\mathbb{R}^N$, $\mathcal{K}$ is a class of Kato potentials with
$$
\|V\|_{\mathcal{K}}:=\sup\limits_{x \in \mathbb{R}^N} \int_{\mathbb{R}^N} \frac{|V(y)|}{|x-y|} d y
$$
and $V_{-}(x):=\min \{V(x), 0\}$ is the negative part of $V$. Suppose $\left\{x_n\right\} \subset \mathbb{R}^N$, we define
$$
\mathcal{L} ^n=-\Delta+V(x-x_n) \quad \text { and } \quad \mathcal{L} ^{\infty}=\left\{\begin{array}{lll}
-\Delta+V(x-x_{\infty}), & \text { if } & x_n \rightarrow x_{\infty} \in \mathbb{R}^N, \\
-\Delta+V_\infty, & \text { if } & \left|x_n\right| \rightarrow \infty.
\end{array}\right.
$$
In particular, $\mathcal{L} \left[\phi\left(x+x_n\right)\right]=\left[\mathcal{L} ^n \phi\right]\left(x+x_n\right)$.
\begin{remark}\label{r1.1}
A typical example of potentials satisfying \eqref{eq1.4} and \eqref{eq1.5} is the following Yukawa-type potential
$$
V(x)=c|x|^{-\sigma} e^{-a|x|}, \quad c \in \mathbb{R}, \quad \sigma \in(0,N-1), \quad a>0 .
$$
The genuine Yukawa potential corresponds to $\sigma=1$. The nonlinear Schr\"odinger equation with Yukawa potential appears in a model describing the interaction between a meson field and a fermion field(see \cite{HYUK1935}). Note that
\begin{equation}\label{eq1.6}
\|V\|_{L^q}=|c|\left[N \alpha(N) (a q)^{q \sigma-N} \Gamma(N-q \sigma)\right]^{\frac{1}{q}}
\end{equation}
and
\begin{equation}\label{eq1.7}
\|V\|_{\mathcal{K}}=2(N-1) \alpha(N-1) |c|a^{\sigma-N+1} \Gamma(N-1-\sigma) ,
\end{equation}
where $\Gamma$ is the Gamma function. In fact,
\begin{eqnarray*}
\|V\|_{L^q}& =&|c|\left(\int_{\mathbb{R}^N}|x|^{-q \sigma} e^{-a q|x|} d x\right)^{\frac{1}{q}} \\
& =&|c|\left(N \alpha(N) \int_0^{\infty} r^{N-1-q \sigma} e^{-a q r} d r\right)^{\frac{1}{q}} \\
& =&|c|\left[N \alpha(N) (a q)^{q \sigma-N} \Gamma(N-q \sigma)\right]^{\frac{1}{q}},
\end{eqnarray*}
which proves \eqref{eq1.6}. To obtain \eqref{eq1.7}, consider
$$
\int_{\mathbb{R}^N} \frac{|V(y)|}{|x-y|} d y=|c| \int_{\mathbb{R}^N} \frac{e^{-a|y|}}{|y|^\sigma|x-y|} d y .
$$
If $x=0$, that is,
$$
\int_{\mathbb{R}^N} \frac{e^{-a|y|}}{|y|^{1+\sigma}} d y=N \alpha(N) \int_0^{\infty} e^{-a r} r^{N-2-\sigma} d r=N \alpha(N) a^{\sigma-N+1} \Gamma(N-1-\sigma) .
$$
If $x \neq 0$, it holds
\begin{equation*}
  \int_{\mathbb{R}^N} \frac{e^{-a|y|}}{|y|^\sigma|x-y|} d y   =\int_0^{\infty} \int_{\mathbb{S}^{N-1}} \frac{e^{-a r}}{r^\sigma|x-r \theta|} r^{N-1} d r d \theta =\int_0^{\infty} e^{-a r} r^{N-2-\sigma} I(x, r) d r,
\end{equation*}
where $r=|y|$ and
$$
I(x, r)=\int_{\mathbb{S}^{N-1}} \frac{1}{\left|\frac{x}{r}-\theta\right|} d \theta .
$$
Take $A \in O(N)$ such that $A e_1=\frac{x}{|x|}$ with $e_1=(1,0,\cdots,0)$, we see that
$$
I(x, r)=\int_{\mathbb{S}^{N-1}} \frac{1}{\left|\frac{|x|}{r} A e_1-\theta\right|} d \theta=\int_{\mathbb{S}^{N-1}} \frac{1}{\left|\frac{|x|}{r} e_1-\theta\right|} d \theta .
$$
By change of variables, we get
$$
\begin{aligned}
I(x, r) & =\int_{-1}^1 \int_{\sqrt{1-s^2 }\mathbb{S}^{N-2}} \frac{d \eta}{\sqrt{\left(\frac{|x|}{r}-s\right)^2+|\eta|^2}} \frac{d s}{\sqrt{1-s^2}} \\
& =\int_{-1}^1 \int_{\mathbb{S}^{N-2}} \frac{\sqrt{1-s^2} d \zeta}{\sqrt{\left(\frac{|x|}{r}-s\right)^2+1-s^2}} \frac{d s}{\sqrt{1-s^2}} \\
& =\left|\mathbb{S}^{N-2}\right| \int_{-1}^1 \frac{d s}{\sqrt{\left(\frac{|x|}{r}-s\right)^2+1-s^2}} \\
& =(N-1)\alpha(N-1) \frac{r}{|x|}\left(\frac{|x|}{r}+1-\left|\frac{|x|}{r}-1\right|\right) \\
& =\left\{\begin{array}{cc}
2(N-1)\alpha(N-1), &  |x| \leq r, \\
2(N-1)\alpha(N-1) \frac{r}{|x|}, &  |x| \geq r .
\end{array}\right.
\end{aligned}
$$
It follows that
\begin{eqnarray*}
&&\int \frac{e^{-a|y|}}{|y|^\sigma|x-y|} d y \\
& =&\frac{2(N-1)\alpha(N-1)}{|x|} \int_0^{|x|} e^{-a r} r^{N-1-\sigma} d r+2(N-1)\alpha(N-1)  \int_{|x|}^{\infty} e^{-a r} r^{N-2-\sigma} d r \\
& =&2(N-1) \alpha(N-1) a^{\sigma-N+1} \Gamma(N-1-\sigma) \\
&&+2(N-1) \alpha(N-1)\left(\frac{1}{|x|} \int_0^{|x|} e^{-a r} r^{N-1-\sigma} d r-\int_0^{|x|} e^{-a r} r^{N-2-\sigma} d r\right) .
\end{eqnarray*}
Consider
$$
f(\lambda)=\frac{1}{\lambda} \int_0^\lambda e^{-a r} r^{N-1-\sigma} d r-\int_0^\lambda e^{-a r} r^{N-2-\sigma} d r, \quad \lambda>0 .
$$
We see that if $0<\sigma<N-1$, then
$$
\lim _{\lambda \rightarrow 0} f(\lambda)=0 .
$$
Moreover,
$$
f^{\prime}(\lambda)=-\frac{1}{\lambda^2} \int_0^\lambda e^{-a r} r^{N-1-\sigma} d r<0, \quad \forall \lambda>0 .
$$
This shows that $f$ is a strictly decreasing function, hence $f(\lambda)<0$ for all $\lambda>0$. Thus for $x \neq 0$,
$$
\int \frac{e^{-a|y|}}{|y|^\sigma|x-y|} d y<2(N-1) \alpha(N-1) a^{\sigma-N+1} \Gamma(N-1-\sigma) .
$$
We conclude that
$$
\|V\|_{\mathcal{K}}=2(N-1) \alpha(N-1) |c|a^{\sigma-N+1} \Gamma(N-1-\sigma),
$$
which proves \eqref{eq1.7}.
\end{remark}

The main results of this paper are as follows.
\begin{theorem}\label{t1.1}
Assume $u_0 \in H_{rad}^1(\mathbb{R}^N)$,\ $N\geq3$. If  $V \in \mathcal{K} \cap L^{\frac{N}{2}}$ and $\left\|V_{-}\right\|_{\mathcal{K}}<N(N-2)\alpha(N)$, then there exists a unique solution $u$ to \eqref{eq1.1} satisfying
 \begin{equation*}
    u \in C(I , H_{rad}^1(\mathbb{R}^N))\cap L^{\frac{2(N+2)}{N-2}}(I,W^{1,{\frac{2N(N+2)}{N^2+4}}})
 \end{equation*}
for $\|u_0\|_{H^1}$ small enough, where $I$ is an interval.
\end{theorem}

\begin{theorem}\label{t1.2}
Let $N \geq 3$ and $u_0 \in H_{\text {rad }}^1$ hold. Assume  $V \in \mathcal{K} \cap L^{\frac{N}{2}}$ and $\left\|V_{-}\right\|_{\mathcal{K}}<N(N-2)\alpha(N)$. Let $u$ be such that the corresponding solution to \eqref{eq1.1} exists on the maximal time $T^*$. If $E(u_0)<0$, then one of the following statements holds true:

$(1)$ $u(t)$ blows up in finite time in the sense that $T^*<+\infty$ must hold.

$(2)$ $u(t)$ blows up infinite time such that
\begin{equation}\label{eq4.1}
  \sup _{t \geq 0}\left\|\nabla u(\cdot, t)\right\|_{L^2}=\infty .
\end{equation}
\end{theorem}

\begin{remark}\label{r1.2}
Under the assumption of Theorem \ref{t1.1}, we can define a maximal interval $I(u_0)=(t_0-T_{-}(u_0),t_0+T_{+}(u_0))$, with $T_{\pm}(u_0)>0$, where the solution is defined. If $T_{+}(u_0)<\infty$, then by standard finite blow-up criterion, we know that
\begin{equation*}
  \|u\|_{L_{[t_0,t_0+T_{+}(u_0)]}^\frac{2(N+2)}{N-2}W^{1,{\frac{2N(N+2)}{N^2+4}}}}=+\infty,
\end{equation*}
the corresponding result holds for $T_{-}(u_0)$.
\end{remark}
\begin{theorem}\label{t1.3}
Assume that $N\geq3$. If $V\in C(\mathbb{R}^N)\cap L^{\frac{N}{2}}(\mathbb{R}^N)$ and $\|V_{-}\|_{\frac{N}{2}}\leq S$, the equation \eqref{eq1.99} has a positive ground state solution which is radially symmetric.
\end{theorem}
\begin{remark}\label{r1.3}
If the operator $-\Delta+V$ is positive definite, then we can use the mountain pass theorem to obtain the ground state solution of the equation. If the operator $-\Delta+V$ is indefinite, we can also obtain the existence of nontrivial solutions(Sobolev subcritical case) via Morse theory, but we can not determine whether this solution is ground state. The critical case is even more complicated, and there are currently no results regarding the ground state solution.
\end{remark}

\begin{theorem}\label{t1.4}
Let  $V \in  C\cap\mathcal{K} \cap W^{1,\frac{N}{2}}$,  $\left\|V_{-}\right\|_{\mathcal{K}}<N(N-2)\alpha(N)$, $\|V_{-}\|_{\frac{N}{2}}\leq S$, $\nabla V(x) \cdot x\in L^{\frac{N}{2}}(\mathbb{R}^N)$ and $\nabla V(x) \cdot x\leq0$, $N\geq3$. Assume that
\begin{equation*}
  E\left(u_0\right)<E(W), \int_{\mathbb{R}^N}\left|\nabla u_0\right|^2dx<\int_{\mathbb{R}^N}|\nabla W|^2dx
\end{equation*}
and $u_0$ is radial. Then there exist $u_{0,+}, u_{0,-}$ in $H^1$ such that
$$
\lim\limits_{t \rightarrow+\infty}\left\|u(t)-e^{i t \mathcal{L}} u_{0,+}\right\|_{H^1}=0, \  \lim\limits_{t \rightarrow-\infty}\left\|u(t)-e^{i t \mathcal{L}} u_{0,-}\right\|_{H^1}=0 .
$$
\end{theorem}
\begin{remark}\label{r1.4}
If the operator $-\Delta+V$ is indefinite and the nonlinear term satisfies the sobolev subcritical growth,  we can obtain a nontrivial solution. However, it is difficult to obtain the scattering results. A very obvious difficulty is that the energy corresponding to this nontrivial solution may not necessarily be positive, so we can not obtain the necessary variational estimates(such as Lemma \ref{L7.1}).
\end{remark}

In section 2, we provide some notations and some important lemma in the proof of main theorems. In the next section, we aim to prove theorem \ref{t1.1}, that is global well-posedness. In sections 4 and 5, we get the blow up solutions in infinite time and long time perturbation result. After that, we prove the existence of positive ground state solution by using variational methods in section 6. Finally, we will obtain the scattering asymptotics. To achieve this goal, we need some new variational estimates and compactness results, that is, sections 7, 8 and 9.

\section{Preliminary}
In this section, we provide some notations and some important lemma in the proof of main theorems.
 \begin{definition}\label{D2.1}
Let $v_0 \in H^1, v(t)=e^{i t \mathcal{L}} v_0$ and let $\left\{t_n\right\}$ be a sequence, with $\lim\limits_{n \rightarrow \infty} t_n=\bar{t} \in[-\infty,+\infty]$. We say that $u(x, t)$ is a non-linear profile associated with $\left(v_0,\left\{t_n\right\}\right)$ if there exists an interval $I$, with $\bar{t} \in I$ (if $\bar{t}= \pm \infty$, $I=[a,+\infty)$ or $(-\infty, a])$ such that $u$ is a solution of \eqref{eq1.1} in $I$ and
$$
\lim\limits_{n \rightarrow \infty}\left\|u\left(-, t_n\right)-v\left(-, t_n\right)\right\|_{H^1}=0 .
$$
\end{definition}
Now, we give some relatively specific definitions that play a crucial role in the proof of concentration compactness in section 8.

\begin{definition}\label{D2.2}
(i) We call scale, every sequence $\mathbf{h}=\left(h_n\right)_{n \geq  0}$ of positive numbers and core, every sequence $\mathbf{z}=\left(z_n\right)_{n \geq  0}=\left(t_n, x_n\right)_{n \geq  0} \subset \mathbb{R} \times \mathbb{R}^N$.

(ii) We say that two pairs $(\mathbf{h}, \mathbf{z})$ and $\left(\mathbf{h}^{\prime}, \mathbf{z}^{\prime}\right)$ are orthogonal if
$$
\frac{h_n}{h_n^{\prime}}+\frac{h_n^{\prime}}{h_n}+\left|\frac{t_n-t_n^{\prime}}{\left(h_n\right)^2}\right|+\left|\frac{x_n-x_n^{\prime}}{h_n}\right| \rightarrow+\infty,\ n \rightarrow \infty .
$$
\end{definition}

\begin{definition}\label{D2.3}
(i) A pair $(q, r)$ is $L^2$-admissible, if $r \in[2,\frac{2N}{N-2})$ and $q$ satisfies
$$
\frac{2}{q}+\frac{N}{r}=\frac{N}{2} .
$$
(ii) A pair $(q, r)$ is $H^1$-admissible, if $r \in[\frac{2N}{N-2},+\infty)$ and $q$ satisfies
$$
\frac{2}{q}+\frac{N}{r}=\frac{N-2}{2} .
$$
\end{definition}

\begin{proposition}\label{P2.1}(see \cite{MKTT1998})
Let $(q, r)$ be an $L^2$-admissible pair. There exists $C=$ $C(r)$, such that
 \begin{equation*}
   \left\|e^{i t \mathcal{L}} h\right\|_{L_t^q L_x^r} \leq  C\| \varphi \|_{L^2\left(\mathbb{R}^N\right)}
 \end{equation*}
for every $\varphi \in L^2(\mathbb{R}^N)$.
\end{proposition}
A direct consequence of Proposition \ref{P2.1}, via Sobolev's inequality, is the following.

\begin{proposition}\label{P2.2}
Let $(q, r)$ be an $H^1$-admissible pair. There exists $C=C(r)$, such that
 \begin{equation*}
   \left\|e^{i t \mathcal{L}} h\right\|_{L_t^q L_x^r} \leq   C\|\nabla \varphi\|_{L^2\left(\mathbb{R}^N\right)}
 \end{equation*}
for every $\varphi \in \dot{H}^1(\mathbb{R}^N)$.
\end{proposition}
\begin{remark}\label{R2.1}
One can actually show: (\cite{MKTT1998}) ii')
$$
\left\|\int_{-\infty}^{+\infty} e^{i(t-\tau)\mathcal{L}} g(-, \tau) d \tau\right\|_{L_t^q L_x^r} \leq C\|g\|_{L_t^{m^{\prime}} L_x^{n^{\prime}}},
$$
where $(q, r),(m, n)$ are any $L^2$-admissible pair.
\end{remark}

\begin{lemma}\label{L2.2}(Sobolev embedding). For $v \in C_0^{\infty}\left(\mathbb{R}^{N+1}\right)$, we have
\begin{equation*}
  \|v\|_{L_t^{\frac{2(N+2)}{N-2}}L_x^{\frac{2(N+2)}{N-2}}}\leq C\|\nabla_xv\|_{L_t^{\frac{2(N+2)}{N-2}}L_x^{\frac{2N(N+2)}{N^2+4}}}(N\geq 3).
\end{equation*}
(Note that $\frac{2(N+2)}{N-2}=q, \frac{2 N(N+2)}{N^2+4}=r$ is admissible.)
\end{lemma}

\begin{lemma}\label{L2.3}(\cite{VDD2018}) Let $F(z)=|z|^k z$ with $k>0, s \geq 0$ and $1<p, p_1<\infty, 1<q_1 \leq \infty$ satisfying $\frac{1}{p}=\frac{1}{p_1}+\frac{k}{q_1}$. If $k$ is an even integer or $k$ is not an even integer with $[s] \leq k$, then there exists $C>0$ such that for all $u \in \mathscr{S}$,
$$
\|F(u)\|_{\dot{W}^{s, p}} \leq C\|u\|_{L^{q_1}}^k\|u\|_{\dot{W}^{s, p_1}} .
$$
A similar estimate holds with $\dot{W}^{s, p}, \dot{W}^{s, p_1}$-norms replaced by $W^{s, p}, W^{s, p_1}$ norms.
\end{lemma}

\begin{lemma}\label{L2.4}(\cite{TBDH2016}) Let $N \geq 1$ and $f: \mathbb{R}^N \rightarrow \mathbb{R}$ satisfy $\nabla f \in W^{1, \infty}(\mathbb{R}^N)$. Then, for all $u \in H^{\frac{1}{2}}(\mathbb{R}^N)$, it holds that
$$
\left|\int_{\mathbb{R}^N} \bar{u}(x) \nabla f(x) \cdot \nabla u(x) d x\right| \leqslant C\left(\left\||\nabla|^{\frac{1}{2}} u\right\|_{L^2}^2+\|u\|_{L^2}\left\||\nabla|^{\frac{1}{2}} u\right\|_{L^2}\right),
$$
with some constant $C>0$ depending only on $\|\nabla f\|_{W^{1, \infty}}$ and $N$.
\end{lemma}
Next, we consider the operators $\mathcal{L} ^{\infty}$ appear as limits of the operators $\mathcal{L} ^n$.

\begin{lemma}\label{L2.5}(see \cite{{RKCM2017},{JLCM2018}}). Suppose $\tau_n \rightarrow \tau_{\infty} \in \mathbb{R}$ and $\left\{x_n\right\} \subset \mathbb{R}^N$ satisfies $x_n \rightarrow x_{\infty} \in \mathbb{R}^N$ or $\left|x_n\right| \rightarrow \infty$. Then,
\begin{eqnarray}\label{eq2.1}
&& \lim _{n \rightarrow \infty}\left\|\mathcal{L}^n \psi-\mathcal{L}^{\infty} \psi\right\|_{\dot{H}^{-1}}=0 \quad \text { for all } \psi \in \dot{H}^1, \nonumber\\
&& \lim _{n \rightarrow \infty}\left\|\left(e^{-i \tau_n \mathcal{L}^n}-e^{-i \tau_{\infty} \mathcal{L}^{\infty}}\right) \psi\right\|_{\dot{H}^{-1}}=0 \quad \text { for all } \quad \psi \in \dot{H}^{-1}, \nonumber\\
&& \lim _{n \rightarrow \infty}\left\|\left[\left(\mathcal{L}^n\right)^{\frac{1}{2}}-\left(\mathcal{L}^{\infty}\right)^{\frac{1}{2}}\right] \psi\right\|_{L_x^2}=0 \quad \text { for all } \quad \psi \in \dot{H}^1 .
\end{eqnarray}
Furthermore, for any $2<q \leq \infty$ and $\frac{2}{q}+\frac{N}{r}=\frac{N}{2}$,
\begin{equation}\label{eq2.2}
\lim _{n \rightarrow \infty}\left\|\left(e^{-i t \mathcal{L}^n}-e^{-i t \mathcal{L}^{\infty}}\right) \psi\right\|_{L_t^q L_x^r(\mathbb{R} \times \mathbb{R}^N)}=0 \quad \text { for all } \quad \psi \in L_x^2 .
\end{equation}
Finally, if $x_{\infty} \neq 0$, then for any $t>0$,
\begin{equation}\label{eq2.3}
\lim _{n \rightarrow \infty}\left\|\left[e^{-t \mathcal{L}^n}-e^{-t \mathcal{L}^{\infty}}\right] \delta_0\right\|_{\dot{H}^{-1}}=0.
\end{equation}
\end{lemma}

\textbf{Notations:}

$\bullet$ Throughout this paper, we use $C$ to denote the universal constant and $C$ may change line by line.

$\bullet$ We also use notation $C(B)$(or $C_B$) to denote a constant depends on $B$.

$\bullet$ We use usual $L^p$ spaces and Sobolev spaces $H^1$. $p^{\prime}$ for the dual index of $p \in(1,+\infty)$ in the sense that $\frac{1}{p^{\prime}}+\frac{1}{p}=1$.

$\bullet$ We use notation $A\Subset B$ to denote $A$ is a open subset of $B$.

\section{Well-posedness}
First, we recall the dispersive estimate for the linear propagator $e^{-i t \mathcal{L}}$, but for simplicity, we assume that the negative part of a potential is small.

\begin{lemma}\label{L3.1} (Dispersive estimate). If $V \in \mathcal{K} \cap L^{\frac{N}{2}}$ and $\left\|V_{-}\right\|_{\mathcal{K}}<N(N-2)\alpha(N)$, then
$$
\left\|e^{-i t \mathcal{L}}u\right\|_{L^{\infty}} \leq C|t|^{-\frac{N}{2}}\|u\|_{L^{1}}
$$
for any $u\in L^1(\mathbb{R}^N)$.
\end{lemma}
\begin{proof}
By Beceanu-Goldberg \cite{MBMG2012}, it suffices to show that $\mathcal{L}$ does not have an eigenvalue or a nonnegative resonance. We claim that $\mathcal{L}$ is positive, that is, if $V \in \mathcal{K}$, then
\begin{equation}\label{eq3.1}
  \int_{\mathbb{R}^N}|V||u|^2 d x \leq  \frac{\|V\|_{\mathcal{K}}}{N(N-2)\alpha(N)} \| \nabla u \|_{L^2}^2  .
\end{equation}
In particular, if $\left\|V_{-}\right\|_{\mathcal{K}}<N(N-2)\alpha(N)$, then
$$
\left(1-\frac{\left\|V_{-}\right\|_{\mathcal{K}}}{N(N-2)\alpha(N)}\right)\|\nabla u\|_{L^2}^2 \leq\|\mathcal{L}^{\frac{1}{2}} u\|_{L^2}^2=\int_{\mathbb{R}^N} \mathcal{L} u \bar{u} d x \leq\left(1+\frac{\|V\|_{\mathcal{K}}}{N(N-2)\alpha(N)}\right)\|\nabla u\|_{L^2}^2 .
$$
In fact, observe that
\begin{eqnarray*}
&&\left\||V|^{\frac{1}{2}}(-\Delta)^{-1}|V|^{\frac{1}{2}} u\right\|_{L^2}^2 \\
& =&\int_{\mathbb{R}^N}|V(x)|\left|\int_{\mathbb{R}^N} \frac{|V(y)|^{\frac{1}{2}}}{N(N-2)\alpha(N)|x-y|} u(y) d y\right|^2 d x \\
& \leq&  \int_{\mathbb{R}^N}|V(x)|\left(\int_{\mathbb{R}^N} \frac{|V(y)|}{N(N-2)\alpha(N)|x-y|} d y\right) \int_{\mathbb{R}^N} \frac{|u(y)|^2}{N(N-2)\alpha(N)|x-y|} d y d x \\
& \leq& \frac{|V|_{\mathcal{K}}}{N(N-2)\alpha(N)} \int_{\mathbb{R}^N} \int_{\mathbb{R}^N} \frac{|V(x)|}{N(N-2)\alpha(N)|x-y|}|u(y)|^2 d y d x \\
& \leq& \left[\frac{|V|_{\mathcal{K}}}{N(N-2)\alpha(N)}\right]^2\|u\|_{L^2}^2 .
\end{eqnarray*}
Then, \eqref{eq3.1} follows by the standard $T T^*$ argument with $T=|V|^{\frac{1}{2}}|\nabla|^{-1}$. Therefore, it has no negative eigenvalue. Moreover, by Ionescu-Jerison \cite{JHSR2008}, there is no positive eigenvalue or resonance.
\end{proof}
\begin{lemma}\label{L3.2}
If  $V \in \mathcal{K} \cap L^{\frac{N}{2}}$ and $\left\|V_{-}\right\|_{\mathcal{K}}<N(N-2)\alpha(N)$, there exist $C_1, C_2>0$ such that
$$
\|f\|_{L^q}\leq C_1\left\|\mathcal{L}^{\frac{s}{2}} f\right\|_{L^p}, \quad\|f\|_{L^q} \leq C_2\left\|(1+\mathcal{L})^{\frac{s}{2}} f\right\|_{L^p},
$$
where $1<p<q<\infty,\ 1<p<\frac{N}{s},\ 0 \leq s \leq 2$ and $\frac{1}{q}=\frac{1}{p}-\frac{s}{N}$.
\end{lemma}
\begin{proof}
Consider the heat operator $e^{-t(a+\mathcal{L})}$. Let $a=0$ or $1$. According to the Theorem 2 in \cite{MT2006}, we know that $e^{-t(a+\mathcal{L})}$ obeys the gaussian heat kernel estimate,
$$
0 \leq  e^{-t(a+\mathcal{L})}(x, y) \leqslant \frac{A_1}{t^{\frac{N}{2}}} e^{-A_2 \frac{|x-y|^2}{t}}, \ \forall t>0,\ \forall x, y \in \mathbb{R}^N
$$
for some $A_1, A_2>0$. Applying it to
$$
(a+\mathcal{L})^{-\frac{s}{2}}=\frac{1}{\Gamma(s)} \int_0^{\infty} e^{-t(a+\mathcal{L})} t^{\frac{s}{2}-1} d t,
$$
it is easy to see that the kernel of $(a+\mathcal{L})^{-\frac{s}{2}}$ satisfies
$$
\left|(a+\mathcal{L})^{-\frac{s}{2}}(x, y)\right| \leq C\frac{1}{|x-y|^{N-s}},
$$
where $C>0$. This implies that $\left\|(a+\mathcal{L})^{-\frac{s}{2}} f\right\|_{L^q} \leq C\|f\|_{L^p}$ with $p, q, s$ in Lemma \ref{L3.2}.
\end{proof}

\begin{lemma}\label{L3.3}  (Norm equivalence)
If  $V \in \mathcal{K} \cap L^{\frac{N}{2}}$ and $\left\|V_{-}\right\|_{\mathcal{K}}<N(N-2)\alpha(N)$, then
$$
\left\|\mathcal{L}^{\frac{s}{2}} f\right\|_{L^r} \sim\|f\|_{\dot{W}^{s, r}}, \quad\left\|(1+\mathcal{L})^{\frac{s}{2}} f\right\|_{L^r} \sim\|f\|_{W^{s, r}},
$$
where $1<r<\frac{N}{s}$ and $0 \leq s \leq 2$.
\end{lemma}
\begin{proof}
Let $a=0$ or $1$. We claim that
$$
\|(a+\mathcal{L}) f\|_{L^r} \sim\|(a+\Delta) f\|_{L^r}, \quad \forall 1<r<\frac{N}{2} .
$$
Indeed, by H\"older's inequality and the Sobolev inequality, we have
$$
\begin{aligned}
\|(a+\mathcal{H}) f\|_{L^r} & \leq \|(a-\Delta) f\|_{L^r}+\|V f\|_{L^r} \\
& \leq \|(a-\Delta) f\|_{L^r}+\|V\|_{L^{\frac{N}{2}} }\|f\|_{L^{\frac{N r}{N-2 r}}} \\
& \leq C\|(a-\Delta) f\|_{L^r} .
\end{aligned}
$$
Similarly,
$$
\begin{aligned}
\|(a-\Delta) f\|_{L^r} & \leq \|(a+\mathcal{L}) f\|_{L^r}+\|V f\|_{L^r} \\
& \leq \|(a+\mathcal{L}) f\|_{L^r}+\|V\|_{L^{\frac{N}{2}}}\|f\|_{L^{\frac{N r}{N-2 r}}} \\
& \leq C\|(a+\mathcal{L}) f\|_{L^r} .
\end{aligned}
$$
Next, we claim that the imaginary power operator $(a+\mathcal{L})^{i y}$ satisfies
$$
\left\|(a-\Delta)^{i y}\right\|_{L^r \rightarrow L^r},\left\|(a+\mathcal{L})^{i y}\right\|_{L^r \rightarrow L^r} \leq C\langle y\rangle^{\frac{N}{2}}, \quad \forall y \in \mathbb{R} \text { and } \forall 1<r<\infty .
$$
Indeed, since the heat kernel operator $e^{-t \mathcal{L}}$ obeys the gaussian heat kernel estimate (see the proof of Lemma \ref{L3.2}, these bounds follow from \cite{ASJW2001}.
Combining the above two claims, we obtain that
$$
\begin{aligned}
\left\|(a+\mathcal{L})^z f\right\|_{L^r} & \leq C\langle\operatorname{Im} z\rangle^{\frac{N}{2}}\left\|(a-\Delta)^z f\right\|_{L^r}, \\
\left\|(a-\Delta)^z f\right\|_{L^r} & \leq C\langle \operatorname{Im} z\rangle^{\frac{N}{2}}\left\|(a+\mathcal{L})^z f\right\|_{L^r}
\end{aligned}
$$
for $1<r<\infty$ when $\operatorname{Re} z=0$ and for $1<r<\frac{N}{2}$ when $\operatorname{Re} z=1$. Finally, applying the Stein-Weiss complex interpolation, we prove the norm equivalence lemma.
\end{proof}

Now, we get the existence and uniqueness of solution to \eqref{eq1.1} by using contraction mapping method.
\begin{lemma}\label{L3.4}
Assume $u_0 \in H_{rad}^1(\mathbb{R}^N)$,\ $N\geq3$, then there exists a unique solution $u$ to \eqref{eq1.1} satisfying
 \begin{equation*}
    u \in C(I , H_{rad}^1(\mathbb{R}^N))\cap L^{\frac{2(N+2)}{N-2}}(I,W^{1,{\frac{2N(N+2)}{N^2+4}}})
 \end{equation*}
for $\|u_0\|_{H^1}$ small enough, where $I$ is an interval.
\end{lemma}
\begin{proof}
\eqref{eq1.1} is equivalent to the integral equation
$$
u(t)=e^{-i t \mathcal{L}} u_0+i\int_0^t e^{-i\left(t-t^{\prime}\right) \mathcal{L}}|u|^{\frac{4}{N-2}}u d t^{\prime}.
$$
Now, we consider  the complete metric space
\begin{equation*}
  \mathbf{B}=\left\{u :\|(1+\mathcal{L})^{\frac{1}{2}}u\|_{L_I^{\frac{2(N+2)}{N-2}}L^{\frac{2N(N+2)}{N^2+4}}}\leq \rho\right\}
\end{equation*}
equipped with the distance
$$
d_{\mathbf{B}}(u, v):=\|u-v\|_{L_I^{\frac{2(N+2)}{N-2}}L^{\frac{2N(N+2)}{N^2+4}}},\ u,v\in\mathbf{B}
$$
and
\begin{equation*}
  \Phi_{u_0}(v)=e^{-i t \mathcal{L}} u_0+i\int_0^t e^{-i\left(t-t^{\prime}\right) \mathcal{L}} |v|^{\frac{4}{N-2}}v d t^{\prime}.
\end{equation*}
Next, we prove that $\Phi_{u_0}(v):  \mathbf{B}\rightarrow  \mathbf{B}$ and is a contraction. In fact, by Lemma \ref{L2.2},
\begin{eqnarray}\label{eq3.2}
&&\left\|(1+\mathcal{L})^{\frac{1}{2}} \Phi_{u_0}(v)\right\|_{L_I^{\frac{2(N+2)}{N-2}}L^{\frac{2N(N+2)}{N^2+4}}}\nonumber\\
&\leq&C\|(1+\mathcal{L})^{\frac{1}{2}}u_0\|_{L^{2}}+\left\|\int_0^t e^{-i\left(t-t^{\prime}\right)\mathcal{L}}(1+\mathcal{L})^{\frac{1}{2}}|v|^{\frac{4}{N-2}}v d t^{\prime}\right\|_{L_I^{\frac{2(N+2)}{N-2}}L^{\frac{2N(N+2)}{N^2+4}}} \nonumber\\
&\leq&C\|(1+\mathcal{L})^{\frac{1}{2}}u_0\|_{L^{2}}+C\left\|(1+\mathcal{L})^{\frac{1}{2}}|v|^{\frac{4}{N-2}}v \right\|_{L_I^{2}L^{\frac{2N}{N+2}}} \nonumber\\
&\leq&C\|u_0\|_{H^1}+C\left\| v\cdot |v|^{\frac{4}{N-2}}  \right\|_{L_I^2 W^{1,\frac{2 N}{N+2}}}   \nonumber\\
&\leq&C\|u_0\|_{H^1} +C\left\|\| v\|_{W^{1,\frac{2N(N+2)}{N^2+4}}}\cdot \|v\|_{L^{\frac{2(N+2)}{N-2}}}^{\frac{4}{N-2}}  \right\|_{L_I^2 }  \nonumber\\
&\leq&C\|u_0\|_{H^1} +C\left\|v\right\|_{L_I^{\frac{2(N+2)}{N-2}}W^{1,{\frac{2N(N+2)}{N^2+4}}}}\cdot\left\| v  \right\|_{L_I^{\frac{2(N+2)}{N-2}}L^{\frac{2(N+2)}{N-2}}}^{\frac{4}{N-2}}\nonumber\\
&\leq&C\|u_0\|_{H^1} +C\left\|v\right\|_{L_I^{\frac{2(N+2)}{N-2}}W^{1,{\frac{2N(N+2)}{N^2+4}}}}^{\frac{N+2}{N-2}}\\
&\leq&C(\|u_0\|_{H^1} +\rho^{\frac{N+2}{N-2}})\nonumber  .
\end{eqnarray}
Choosing $\rho$ such that $C(\|u_0\|_{H^1} +\rho^{\frac{N+2}{N-2}})\leq\rho$, then $\Phi_{u_0}(v)$ maps $\mathbf{B}$ to $\mathbf{B}$.

Now we prove that $\Phi_{u_0}(v)$ is a contraction map. Let $u, v \in \mathbf{B}$, by Remark \ref{R2.1} and Lemma \ref{L2.3}, then we have
 \begin{eqnarray*}
 &&d_B(\Phi_{u_0}(u), \Phi_{u_0}(v))\nonumber\\
 &=&\|\Phi_{u_0}(u)-\Phi_{u_0}(v)\|_{L_I^\frac{2(N+2)}{N-2}L^{\frac{2N(N+2)}{N^2+4}}}\nonumber\\
 &\leq&\left\|\int_0^t e^{-i\left(t-t^{\prime}\right) \mathcal{L}}|u\cdot |u|^{\frac{4}{N-2}}-v\cdot |v|^{\frac{4}{N-2}}| d t^{\prime}\right\|_{L_I^\frac{2(N+2)}{N-2}L^{\frac{2N(N+2)}{N^2+4}}} \nonumber\\
 &\leq&C\|u\cdot |u|^{\frac{4}{N-2}}-v\cdot |v|^{\frac{4}{N-2}}\|_{L_I^{2} L^{\frac{2N}{N+2}}}\nonumber\\
 &\leq&C\| u-v\|_{L_I^\frac{2(N+2)}{N-2}L^{\frac{2N(N+2)}{N^2+4}}} \left(\left\|u\right\|_{L_I^{\frac{2(N+2)}{N-2}}L^{\frac{2N(N+2)}{N^2+4}}}^{\frac{4}{N-2}}+\left\|v\right\|_{L_I^{\frac{2(N+2)}{N-2}}L^{\frac{2N(N+2)}{N^2+4}}}^{\frac{4}{N-2}}\right)\\ &\leq&C\| u-v\|_{L_I^\frac{2(N+2)}{N-2}L^{\frac{2N(N+2)}{N^2+4}}} \left(\left\|u\right\|_{L_I^{\frac{2(N+2)}{N-2}}W^{1,{\frac{2N(N+2)}{N^2+4}}}}^{\frac{4}{N-2}}+\left\|v\right\|_{L_I^{\frac{2(N+2)}{N-2}}W^{1,{\frac{2N(N+2)}{N^2+4}}}}^{\frac{4}{N-2}}\right)\\
 &\leq&C\rho^{\frac{4}{N-2}} d_\mathbf{B}(u,v).
 \end{eqnarray*}
 Choosing $\rho$ such that $C \rho^{\frac{4}{N-2}} \leq\frac{1}{2}$, the above estimate implies that $\Phi_{u_0}(v)$ is a contraction. Therefore, $\Phi_{u_0}(v)$ has a fixed point in $\mathbf{B}$.
\end{proof}

\begin{proof}[\bf Proof of Theorem \ref{t1.1}]
The above lemma implies global well-posedness for \eqref{eq1.1}.
\end{proof}

\section{Blow up solutions}
In this section, we will investigate the blow up solutions of \eqref{eq1.1}.  Let $\psi \in C_0^{\infty}(\mathbb{R}^N)$ be radial and satisfy
$$
\psi(r)=\left\{\begin{array}{ll}
\frac{1}{2} r^2, & \text { for } r \leqslant 1 \\
0, & \text { for } r \geqslant 2
\end{array} \quad \text { and } \quad \psi^{\prime \prime}(r) \leqslant 1 \quad \text { for } r=|x| \geqslant 0 .\right.
$$
For a fixed $R>0$, we define the rescaled function $\psi_R: \mathbb{R}^N \rightarrow \mathbb{R}$ by setting
 \begin{equation}\label{eq4.2}
\psi_R(r):=R^2 \psi\left(\frac{r}{R}\right) .
 \end{equation}
Next we will show that
 \begin{equation}\label{eq4.3}
   1-\psi_R^{\prime \prime}(r) \geqslant 0, \quad 1-\frac{\psi_R^{\prime}(r)}{r} \geqslant 0, \quad N-\Delta \psi_R(r) \geq  0 \quad \text { for all } r \geqslant 0 .
 \end{equation}
Indeed, this first inequality follows from $\psi_R^{\prime \prime}(r)=\psi^{\prime \prime}(\frac{r}{ R}) \leq 1$. We obtain the second inequality by integrating the first inequality on $[0, r]$ and using that $\psi_R^{\prime}(0)=0$. Finally, we see that last inequality follows from
$$
N-\Delta \psi_R(r)=1-\psi_R^{\prime \prime}(r)+ (N-1)\left(1-\frac{1}{r} \psi_R^{\prime}(r)\right) \geqslant 0 .
$$
Besides \eqref{eq4.3}, $\psi_R$ admits the following properties, which can be easily checked.
We define
$$
\mathcal{M}_{\psi_R}[u(t)]:=2\mathrm{Im}  \int_{\mathbb{R}^N} \overline{u}(t) \nabla \psi_R \cdot \nabla u(t) d x=2\mathrm{Im} \int_{\mathbb{R}^N} \overline{u}(t) \partial_j \psi_R \partial_j u(t) d x .
$$
Define the self-adjoint differential operator
$$
\Gamma_{\psi_R}:=i\left(\nabla \cdot \nabla \psi_R+\nabla \psi_R \cdot \nabla\right),
$$
which acts on functions according to
$$
\Gamma_{\psi_R} f= i\left(\nabla \cdot\left(\left(\nabla \psi_R\right) f\right)+\left(\nabla \psi_R\right) \cdot(\nabla f)\right) .
$$
It's easy to check that
$$
\mathcal{M}_{\psi_R}[u(t)]=\left\langle u(t), \Gamma_{\psi_R} u(t)\right\rangle .
$$

Next, we show the following useful lemma.
\begin{lemma}\label{L4.1}
Let $N \geq 3$, and $u \in H_{\text {rad }}^1$ is a solution of \eqref{eq1.1}. Let $\psi_R$ be as in \eqref{eq4.2}, $T^*$ be the maximal existence time of solution $u(t)$ in $H_{\text {rad }}^1$. Then for sufficiently large $R$, it holds
$$
\frac{d}{d t} \mathcal{M}_{\psi_R}[u(t)] \leq 8 E(u(t)), \quad t \in\left[0, T^*\right) .
$$
\end{lemma}
\begin{proof}
By taking the derivative of $\mathcal{M}_{\psi_R}[u(t)]$ with respect to time $t$ and using the equation of $u(t)$, for any $t \in[0, T)$, it follows that
 \begin{eqnarray*}
\frac{d}{d t} \mathcal{M}_{\psi_R}[u(t)]&= & \left\langle u(t),\left[-\Delta,  i\Gamma_{\psi_R}\right] u(t)\right\rangle+\left\langle  -|u|^{\frac{4}{N-2}}u, i \Gamma_{\psi_R}  u(t)\right\rangle+\left\langle u(t) , i \Gamma_{\psi_R} |u|^{\frac{4}{N-2}}u  \right\rangle \\
&= & I_1+I_2+I_3,
 \end{eqnarray*}
 where $[X, Y] \equiv X Y-Y X$ denotes the commutator of operators $X$ and $Y$.
According to the localized radial virial estimate in \cite{TBDH2016}, we obtain
\begin{eqnarray*}
I_1&=&\left\langle u(t),\left[-\Delta, i \Gamma_{\psi_R}\right] u(t)\right\rangle \leq 8 \left\|\nabla u\right\|_{L^2}^2+C R^{-2  },\\
  I_2 &=& \left\langle  -|u|^{\frac{4}{N-2}}u, i \Gamma_{\psi_R}  u(t)\right\rangle \\
    &=&\left\langle   |u|^{\frac{4}{N-2}}u, \left(\nabla \cdot\left(\left(\nabla \psi_R\right) u\right)+\left(\nabla \psi_R\right) \cdot(\nabla u)\right)\right\rangle\\
    & \leq&\frac{4}{N} \int_{\mathbb{R}^N}  \nabla \psi_R \cdot \nabla\left(|u|^{\frac{2N}{N-2}}\right) d x \\
    & =&-\frac{4}{N} \int_{\mathbb{R}^N}\left(\Delta \psi_R\right) |u|^{\frac{2N}{N-2}} dx\\
    & =&- 4 \int_{\mathbb{R}^N} |u|^{\frac{2N}{N-2}}dx-\frac{4}{N} \int_{|x| \geq  R}(\Delta \psi_R-N)|u|^{\frac{2N}{N-2}}dx\\
    &=&-  4 \int_{\mathbb{R}^N} |u|^{\frac{2N}{N-2}}dx+\frac{4}{N} \int_{|x| \geq  R}(N-\Delta \psi_R) |u|^{\frac{2N}{N-2}} dx\\
    & \leq&- 4\int_{\mathbb{R}^N} |u|^{\frac{2N}{N-2}}dx+2CR^{-2}\|N-\Delta \psi_R\|_{L^\infty}\|\nabla u\|_{L^2(|x|\geq R)}^2,\\
I_3 &=& \left\langle u(t) , i \Gamma_{\psi_R} |u|^{\frac{4}{N-2}}u   \right\rangle \\
    &=&-\left\langle  u(t), \left(\nabla \cdot\left(\left(\nabla \psi_R\right) |u|^{\frac{4}{N-2}}u\right)+\left(\nabla \psi_R\right) \cdot(\nabla |u|^{\frac{4}{N-2}}u)\right)\right\rangle\\
    & \leq&\frac{4}{N} \int_{\mathbb{R}^N}  \nabla \psi_R \cdot \nabla\left(|u|^{\frac{2N}{N-2}}\right) d x \\
    & =&-\frac{4}{N} \int_{\mathbb{R}^N}\left(\Delta \psi_R\right) |u|^{\frac{2N}{N-2}} dx\\
    &=&-  4 \int_{\mathbb{R}^N} |u|^{\frac{2N}{N-2}}dx+\frac{4}{N} \int_{|x| \geq  R}(N-\Delta \psi_R) |u|^{\frac{2N}{N-2}} dx\\
    & \leq&-4\int_{\mathbb{R}^N} |u|^{\frac{2N}{N-2}}dx+2CR^{-2}\|N-\Delta \psi_R\|_{L^\infty}\|\nabla u\|_{L^2(|x|\geq R)}^2.
\end{eqnarray*}
Therefore, by Lemma \ref{L3.3},
\begin{eqnarray*}
  &&\frac{d}{d t} \mathcal{M}_{\psi_R}[u(t)] \\
  &\leq& 8 \left\|\nabla u\right\|_{L^2}^2+C R^{-2  }- 8\int_{\mathbb{R}^N} |u|^{\frac{2N}{N-2}}dx+4CR^{-2}\|N-\Delta \psi_R\|_{L^\infty}\|\nabla u\|_{L^2(|x|\geq R)}^2\\
  &\leq&8C\left\|\mathcal{L}^{\frac{1}{2}}  u\right\|_{L^2}+C R^{-2  }- 8\int_{\mathbb{R}^N} |u|^{\frac{2N}{N-2}}dx+4CR^{-2}\|N-\Delta \psi_R\|_{L^\infty}\|\nabla u\|_{L^2(|x|\geq R)}^2,
\end{eqnarray*}
where the constant $C>0$ is independent of $R$. When $R>1$ is sufficiently large, then
$$
\frac{d}{d t} \mathcal{M}_{\psi_R}[u(t)] \leq 8CE(u(t))=8 C  E(u_0) .
$$
\end{proof}

\begin{proof}[\bf Proof of Theorem \ref{t1.2} ]
Let $u$ be such that the corresponding solution to \eqref{eq1.1} exists on the maximal time $T^*$. If $T^*<\infty$, then we are done. If $T^*=\infty$, we show \eqref{eq4.1}. We suppose that $u(t)$ exists for all times $t \geq 0$, i.e. $T^*=\infty$.
It follows from Lemma \ref{L4.1} and conservation of mass, for $R \gg 1$ large enough,
$$
\frac{d}{d t} \mathcal{M}_{\psi_R}[u(t)] \leq 8 C  E(u_0):=-A^*<0, \quad t \geq 0 .
$$
From this, we infer that
$$
\mathcal{M}_{\psi_R}[u(t)] \leq-A^* t+\mathcal{M}_{\psi_R}[u_0], \quad t \geq 0 .
$$
On the one hand, let $T_0=\frac{2\left|\mathcal{M}_{\psi_R}[u_0]\right|}{A^*}>0$, then for any $t \geq T_0$, we have
 \begin{equation}\label{eq4.4}
   \mathcal{M}_{\psi_R}[u(t)] \leq-\frac{1}{2} A^* t<0 .
 \end{equation}
On the other hand, by Lemma \ref{L2.4} and the conservation of mass, we see that for any $t \in[0,+\infty)$,
\begin{eqnarray*}
\left|\mathcal{M}_{\psi_R}[u(t)]\right| & \leq& C\left(\psi_R\right)\left(\left\||\nabla|^{\frac{1}{2}} u(t)\right\|_{L^2}^2+\|u(t)\|_{L^2}\left\||\nabla|^{\frac{1}{2}} u(t)\right\|_{L^2}\right) \\
& \leq& C\left(\psi_R\right)\left(\left\|\nabla u\right\|_{L^2} \|u\|_{L^2}+\|u(t)\|_{L^2}^{\frac{3}{2}}\left\|\nabla u\right\|_{L^2}^{\frac{1}{2}}\right) \\
& \leq& C\left(\psi_R\right) \left\|\nabla u\right\|_{L^2}^2,
\end{eqnarray*}
where we have used the interpolation estimate
$$
\left\||\nabla|^{\frac{1}{2}} u\right\|_{L^2}^2 \leq C\left\|\nabla u\right\|_{L^2} \|u\|_{L^2}  .
$$
This combined with \eqref{eq4.4} yields that for any $t \geq T_0$,
$$
A^* t \leq-2 \mathcal{M}_{\psi_R}[u(t)] \leq C\left\|\nabla u\right\|_{L^2}^2.
$$
This shows that
$$
\left\|\nabla u(t)\right\|_{L^2}^2 \geq C t , \quad t \geq T_0 .
$$
It means that
$$
\sup\limits_{t \geq 0}\left\|\nabla u(\cdot, t)\right\|_{L^2}=\infty .
$$
\end{proof}

\section{Perturbation theory}
In this section, we will study the long-time perturbation theory.

\begin{proposition}(Perturbation theory)\label{P5.1}
Let $\tilde{u}: I \times \mathbb{R}^N \rightarrow \mathbb{C}$ be a solution to the perturbed Schr\"odinger
equation with general nonlinearity
$$
i \partial_t \tilde{u}+\Delta \tilde{u} -V(x)\tilde{u} +|\tilde{u}|^{\frac{4}{N-2}}\tilde{u} =e
$$
for some function $e$. Suppose that
\begin{equation}\label{eq5.1}
  \|\tilde{u}\|_{L^{\infty} H^1\left(I \times \mathbb{R}^N\right)}  \leq E,
\end{equation}
\begin{equation}\label{eq5.2}
  \|\tilde{u}\|_{L^{\frac{2(N+2)}{N-2}}L^{\frac{2(N+2)}{N-2}} (I \times \mathbb{R}^N)}   \leq L
\end{equation}
for some $E, L>0$. Let $u_0 \in H^1(\mathbb{R}^N)$ with $\left\|u_0\right\|_{L^2(\mathbb{R}^N)} \leq M$ for some $M>0$ and let $t_0 \in I$. There exists $\varepsilon_0=\varepsilon_0(E, L, M)>0$ such that if
\begin{equation}\label{eq5.3}
  \left\|u_0-\tilde{u}\left(t_0\right)\right\|_{H^1}  \leq \varepsilon,
\end{equation}
\begin{equation}\label{eq5.4}
  \left\|e\right\|_{L_I^\infty H^1\cap L_I^2 L^{\frac{2 N}{N+2}}}   \leq \varepsilon
\end{equation}
for $0<\varepsilon<\varepsilon_0$, then the unique global solution $u$ to \eqref{eq1.1} with $u\left(t_0\right)=u_0$ satisfies
\begin{equation}\label{eq5.5}
  \left\|u-\tilde{u}\right\|_{L_I^\infty H^1\cap L_I^\frac{2(N+2)}{N-2}W^{1,{\frac{2N(N+2)}{N^2+4}}}} \leq C(E, L, M) \varepsilon,
\end{equation}
where $C(E, L, M)$ is a non-decreasing function of $E, L$ and $M$.
\end{proposition}
\begin{proof}
Without loss of generality, we may assume $t_0=0 \in I$.  From Theorem \ref{t1.1} we know that $u$ exists globally and
\begin{equation*}
  \|u\|_{L_I^\frac{2(N+2)}{N-2}W^{1,{\frac{2N(N+2)}{N^2+4}}}}\leq\rho,
\end{equation*}
so we need to get \eqref{eq5.5}.

Let
\begin{equation*}
   w:=u-\tilde{u}  \ \text{and}\  A(t):=\|w\|_{L_{(I \cap[-T, T])}^\infty H^1\cap L_{(I \cap[-T, T])}^\frac{2(N+2)}{N-2}W^{1,{\frac{2N(N+2)}{N^2+4}}}}.
\end{equation*}
Note that
\begin{equation*}
  i \partial_t \tilde{w}+\Delta \tilde{w}-V(x)\tilde{w} +|\tilde{u}+w|^{\frac{4}{N-2}}(\tilde{u}+w)-|\tilde{u}|^{\frac{4}{N-2}}\tilde{u} =-e,
\end{equation*}
which is equivalent to the integral equation
$$
w(t)=e^{i t \mathcal{L}} w_0+i\int_0^t e^{i\left(t-t^{\prime}\right)  \mathcal{L}} [|\tilde{u}+w|^{\frac{4}{N-2}}(\tilde{u}+w)-|\tilde{u}|^{\frac{4}{N-2}}\tilde{u}] d t^{\prime}+i\int_0^t e^{i\left(t-t^{\prime}\right)  \mathcal{L}}e d t^{\prime}.
$$
Then by Strichartz, H\"older, \eqref{eq5.1}, \eqref{eq5.2}, \eqref{eq5.3}, \eqref{eq5.4}, we obtain
\begin{eqnarray*}\
&&\left\|(1+\mathcal{L})^{\frac{1}{2}}w(t)\right\|_{L_{(I \cap[-T, T])}^\infty L^2\cap L_{(I \cap[-T, T])}^\frac{2(N+2)}{N-2}L^{ \frac{2N(N+2)}{N^2+4}} }\nonumber\\
&=&C\|(1+\mathcal{L})^{\frac{1}{2}}w_0\|_{L^{2}}+\left\| i\int_0^t e^{i\left(t-t^{\prime}\right) \mathcal{L}}(1+\mathcal{L})^{\frac{1}{2}} e d t^{\prime}\right\|_{L_{(I \cap[-T, T])}^\infty L^2}\nonumber\\
&&+\left\| i\int_0^t e^{i\left(t-t^{\prime}\right) \mathcal{L}} (1+\mathcal{L})^{\frac{1}{2}} [|\tilde{u}+w|^{\frac{4}{N-2}}(\tilde{u}+w)-|\tilde{u}|^{\frac{4}{N-2}}\tilde{u}] d t^{\prime} \right\|_{  L_t^\frac{2(N+2)}{N-2}L^{ \frac{2N(N+2)}{N^2+4}} }\nonumber\\
&\leq&C\|w_0\|_{H^1} +C\left\|e\right\|_{L_T^1H^1}+C\left\|(1+\mathcal{L})^{\frac{1}{2}}\left[|w|\cdot( |\tilde{u}+w|^ {\frac{4}{N-2}}+|\tilde{u}|^ {\frac{4}{N-2}}) \right] \right\|_{L_T^2 L^{ \frac{2 N}{N+2}} } \nonumber\\
&\leq&C\|w_0\|_{H^1} +T\left\|e\right\|_{L_T^\infty H^1}+C\left\|(1+\mathcal{L})^{\frac{1}{2}}\left(|w|\cdot  |\tilde{u}+w|^ {\frac{4}{N-2}}  \right)\right\|_{L_T^2 L^{\frac{2 N}{N+2}}}  \nonumber\\
&&+C\left\|(1+\mathcal{L})^{\frac{1}{2}}\left(|w|\cdot |\tilde{u}|^ {\frac{4}{N-2}}\right) \right\|_{L_T^2 L^{ \frac{2 N}{N+2}}} \nonumber\\
&\leq&C\|w_0\|_{H^1}+T\left\|e\right\|_{L_T^\infty H^1}+C\left\|w\right\|_{L_T^{\frac{2(N+2)}{N-2}}W^{1,{\frac{2N(N+2)}{N^2+4}}}}\cdot\left\| \tilde{u}  \right\|_{L_T^{\frac{2(N+2)}{N-2}}L^{\frac{2(N+2)}{N-2}}}^{\frac{4}{N-2}}\nonumber\\
&&+C\left\|w\right\|_{L_T^{\frac{2(N+2)}{N-2}}W^{1,{\frac{2N(N+2)}{N^2+4}}}}\cdot\left\| \tilde{u} +w \right\|_{L_T^{\frac{2(N+2)}{N-2}}L^{\frac{2(N+2)}{N-2}}}^{\frac{4}{N-2}}\\
&\leq&C\|w_0\|_{H^1}+T\left\|e\right\|_{L_T^\infty H^1}+C\left\|w\right\|_{L_T^{\frac{2(N+2)}{N-2}}W^{1,{\frac{2N(N+2)}{N^2+4}}}}\cdot\left\| \tilde{u}  \right\|_{L_T^{\frac{2(N+2)}{N-2}}L^{\frac{2(N+2)}{N-2}}}^{\frac{4}{N-2}}\nonumber\\
&&+C\left\|w\right\|_{L_T^{\frac{2(N+2)}{N-2}}W^{1,{\frac{2N(N+2)}{N^2+4}}}}\cdot\rho^{\frac{4}{N-2}}\\
&\leq& C\varepsilon +T\varepsilon+CA(t)L^{\frac{4}{N-2}}+CA(t)\rho^{\frac{4}{N-2}},
\end{eqnarray*}
where all space time norms are over $(I \cap[-t, t]) \times \mathbb{R}^N$. Using the standard continuity argument to remove the restriction to $[-t, t]$, we derive \eqref{eq5.5}.
\end{proof}

\begin{remark}\label{r4.1}
Note that $f(u) \in L_T^\infty H^1\cap L_T^\frac{2(N+2)}{N-2}W^{1,{\frac{2N(N+2)}{N^2+4}}}$ and hence
\begin{equation*}
   \left\|\int_t^{\infty} e^{i\left(t-t^{\prime}\right) \mathcal{L}} f(u) d t^{\prime}\right\|_{H^1} \rightarrow 0,\  t \rightarrow+\infty.
\end{equation*}
Then, $u(t)=e^{i(t-a) \mathcal{L}} u_0+\int_a^t e^{i\left(t-t^{\prime}\right) \mathcal{L}} f(u) d t^{\prime}$ and hence $u^{+}=e^{-i a \mathcal{L}} u_0+\int_a^{\infty} e^{-i t^{\prime} \mathcal{L}} f(u) d t^{\prime}$ has the desired property. In fact note that the argument used at the beginning of the proof of Proposition \ref{P5.1} shows that it suffices to assume $u$ to be a solution of \eqref{eq1.1} in $I^{\prime} \times \mathbb{R}^N, I^{\prime} \Subset I$, such that $\|u\|_{L_{I}^{\frac{2(N+2)}{N-2}}W^{1,{\frac{2N(N+2)}{N^2+4}}}}<\infty$.
\end{remark}
\section{Positive ground state solution}
In this section, we will prove the existence of standing wave solutions of \eqref{eq1.1}. More precisely, we can obtain a positive ground state solution which is radially symmetric. Obviously, solutions of problem \eqref{eq1.99} can be obtained by looking for critical points of the functional $I: H^1(\mathbb{R}^N)\rightarrow\mathbb{R}$ defined by
\begin{equation*}
  I(u)= \frac{1}{2} \int_{\mathbb{R}^N}(\left|\nabla u\right|^2+V(x)|u|^2) d x-   \frac{1}{2^*}\int_{\mathbb{R}^N}|u|^{2^*} d x,
\end{equation*}
where $2^*=\frac{2N}{N-2}$. In order to state our main results, we introduce some notations. The Aubin-Talenti constant \cite{TA1976} is denoted by $S$, that is, $S$ is the best constant in the Sobolev embedding $\mathcal{D}^{1,2}(\mathbb{R}^N) \hookrightarrow L^{2^*}(\mathbb{R}^N)$, where $\mathcal{D}^{1,2}(\mathbb{R}^N)$ denotes the completion of $C_c^{\infty}(\mathbb{R}^N)$ with respect to the norm $\|u\|_{E}:=$ $\|\nabla u\|_2$. It is well known \cite{GT1976} that the optimal constant is achieved by (any multiple of)
\begin{equation}\label{eq6.1}
  U_{\varepsilon, y}(x)=[N(N-2)]^{\frac{N-2}{4}}\left(\frac{\varepsilon}{\varepsilon^2+|x-y|^2}\right)^{\frac{N-2}{2}},\ \varepsilon>0,\ y \in \mathbb{R}^N,
\end{equation}
which are the only positive classical solutions to the critical Lane-Emden equation
$$
-\Delta w=w^{2^*-1}, \quad w>0 \quad \text { in } \mathbb{R}^N .
$$

Now, we begin proving that $I$ satisfies the geometric assumptions of the mountain pass theorem.
\begin{lemma}\label{L6.1}
$I$ has a mountain pass geometry, that is,

$\mathrm{(i)}$ there exist $\alpha,\rho>0$ such that $I(v)\geq\alpha$ for all $v\in E$ such that $\|v\|_E =\rho$,

$\mathrm{(ii)}$ there exists $e\in E$ with $\|e\|_{E}>\rho$ such that $I(e)<0$.
\end{lemma}
\begin{proof}
$\mathrm{(i)}$ By using Sobolev embedding, we obtain
 \begin{eqnarray*}
I(u)&=& \frac{1}{2} \int_{\mathbb{R}^N}(\left|\nabla u\right|^2+V(x)|u|^2) d x-   \frac{1}{2^*}\int_{\mathbb{R}^N}|u|^{2^*} d x\\
&\geq&\frac{1}{2}\left(1-\left\|V_{-}\right\|_{\frac{N}{2}} S^{-1}\right) \int_{\mathbb{R}^N}|\nabla u|^2 d x-\frac{1}{2^*\cdot S^{
\frac{2^*}{2}}}\left(\int_{\mathbb{R}^N}|\nabla u|^2 d x\right)^{\frac{2^*}{2}}\\
&>&0
 \end{eqnarray*}
when $\|u\|_E$ is small enough. Hence, there exist $\alpha,\rho>0$ such that $I(v)\geq\alpha$ for all $v\in E$ such that $\|v\|_E=\rho$.

$\mathrm{(ii)}$ Taking $v_0(x)\in C_0^\infty(\mathbb{R}^N)$, we get
\begin{eqnarray*}
I(tv_0(x))&=& \frac{t^2}{2} \int_{\mathbb{R}^N}(\left|\nabla v_0\right|^2+V(x)|v_0|^2) d x-   \frac{t^{2^*}}{2^*}\int_{\mathbb{R}^N}|v_0|^{2^*} d x\\
&\leq&\frac{t^2}{2}\left(1+\left\|V \right\|_{\frac{N}{2}} S^{-1}\right) \int_{\mathbb{R}^N}|\nabla v_0|^2 d x-    \frac{t^{2^*}}{2^*}\int_{\mathbb{R}^N}|v_0|^{2^*} d x \\
&<&0
\end{eqnarray*}
when $t$ is large enough. Therefore, we can choose $e=t_0v_0(x)$ for some $t_0>0$ such that $\mathrm{(ii)}$ holds.
\end{proof}

According to Lemma \ref{L6.1} and using a variant of the mountain pass theorem without the Palais-Smale condition, we can find a sequence $\{v_{n}\} \subset E$ such that
\begin{equation}\label{eq6.2}
  I\left(v_{n}\right) \rightarrow c \  \text {and} \  I^{\prime}\left(v_{n}\right) \rightarrow 0 \text { in } E^{*}
\end{equation}
as $n \rightarrow \infty$, where
$$
c:=\inf _{\gamma \in \Gamma} \max _{t \in[0,1]} I(\gamma(t))
$$
and
$$
\Gamma:=\left\{\gamma \in C\left([0,1], E\right): \gamma(0)=0, I(\gamma(1))<0\right\}.
$$
\begin{lemma}\label{L6.2}
Suppose that $\{u_{n}\} \subset E$ satisfies \eqref{eq6.2}, then $\{u_{n}\}$ is bounded in $E$.
\end{lemma}
\begin{proof}
For $n$ large enough, we obtain that
\begin{eqnarray*}
  c+1+\|u_n\|_E &=& I(u_n)-\frac{1}{2^*} \langle I^{\prime}(u_{n}), u_n\rangle \\
    &=&\frac{1}{2} \int_{\mathbb{R}^N}(\left|\nabla u_n\right|^2+V(x)|u_n|^2) d x-   \frac{1}{2^*}\int_{\mathbb{R}^N}|u_n|^{2^*} d x   \\
    &&-\frac{1}{2^*} \left[\int_{\mathbb{R}^N}(\left|\nabla u_n\right|^2+V(x)|u_n|^2) d x-   \int_{\mathbb{R}^N}|u_n|^{2^*} d x\right]\\
    &=& \frac{1}{N} \int_{\mathbb{R}^N}(\left|\nabla u_n\right|^2+V(x)|u_n|^2) d x\\
    &\geq&\frac{1}{2}\left(1-\left\|V_{-}\right\|_{\frac{N}{2}} S^{-1}\right)\|u_n\|_E^2,
\end{eqnarray*}
which implies that $\{u_{n}\}$ is bounded.
\end{proof}

\begin{lemma}\label{L6.3}
$c<\frac{1}{N}S^{\frac{N}{2}}$.
\end{lemma}
\begin{proof}
Let $\psi \in C_0^\infty(\mathbb{R}^N)$ be a radial cut-off function such that $\psi \equiv 1$ for $x\in B(0, \rho)$ and $\psi \equiv 0$ for $x\in \mathbb{R}^N\setminus B(0, 2\rho)$, where $\rho$ is some positive constant. Define, for $\varepsilon>0$,
$$
\begin{aligned}
& U_{\varepsilon}(x):=\varepsilon^{\frac{2-N}{2}} U(\frac{x}{\varepsilon}), \\
& u_{\varepsilon}(x):=\psi(x) U_{\varepsilon}(x) .
\end{aligned}
$$
It follows from \cite{MW1996} that
$$
\|\nabla U_{\varepsilon}\|_2^2=\|U_{\varepsilon}\|_{2^*}^2=S^{\frac{N}{2}}
$$
and
\begin{eqnarray*}
\int_{\mathbb{R}^N}\left|\nabla u_{\varepsilon}\right|^2dx & =&\int_{\mathbb{R}^N}\left|\nabla U_{\varepsilon}\right|^2dx+O\left(\varepsilon^{N-2}\right)=S^{\frac{N}{2}}+O\left(\varepsilon^{N-2}\right), \\
\int_{\mathbb{R}^N}\left|u_{\varepsilon}\right|^{2^*}dx
& =&\int_{\mathbb{R}^N}\left|U_{\varepsilon}\right|^{2^*}dx+O\left(\varepsilon^N\right)=S^{\frac{N}{2}}+O\left(\varepsilon^N\right), \\
\int_{\mathbb{R}^N}\left|u_{\varepsilon}\right|^2dx
& \geq& \begin{cases}d \left(\int_0^2 \psi(r) d r\right) \varepsilon+O\left(\varepsilon^2\right), & \text { if } N=3,\\
d \varepsilon^2|\ln \varepsilon|+O\left(\varepsilon^2\right), & \text { if } N=4, \\
d \varepsilon^2+O\left(\varepsilon^{N-2}\right), & \text { if } N \geq 5,\end{cases}
\end{eqnarray*}
where $d$ is a positive constant. Note that
\begin{eqnarray*}
I(tu_{\varepsilon})&=& \frac{t^2}{2} \int_{\mathbb{R}^N}(\left|\nabla u_{\varepsilon}\right|^2+V(x)|u_{\varepsilon}|^2) d x-   \frac{t^{2^*}}{2^*}\int_{\mathbb{R}^N}|u_{\varepsilon}|^{2^*} d x\\
&:=&h(t)+\frac{t^2}{2} \int_{\mathbb{R}^N} V(x)|u_{\varepsilon}|^2  d x.
\end{eqnarray*}
Clearly, $h(t)>0$ for $t>0$ small and $h(t) \rightarrow-\infty$ as $t \rightarrow \infty$, so $h(t)$ attains its maximum at
\begin{equation*}
   t_\varepsilon=\left(\frac{  \|\nabla u_\varepsilon\|_2^2}{\|u_\varepsilon\|_{2^*}^{2^*}}\right)^{\frac{1}{2^*-2}} \text{ with } h^{\prime}(t_\varepsilon)=0.
\end{equation*}
Hence, we have
\begin{eqnarray*}
&&I(t_\varepsilon u_{\varepsilon})\\
&=& \frac{t_\varepsilon^2}{2} \int_{\mathbb{R}^N}(\left|\nabla u_{\varepsilon}\right|^2+V(x)|u_{\varepsilon}|^2) d x-   \frac{t_\varepsilon^{2^*}}{2^*}\int_{\mathbb{R}^N}|u_{\varepsilon}|^{2^*} d x\\
&\leq&\frac{1}{2}\left(\frac{ \|\nabla u_\varepsilon\|_2^2}{\|u_\varepsilon\|_{2^*}^{2^*}}\right)^{\frac{2}{2^*-2}}\|\nabla u_\varepsilon\|_2^2-\frac{1}{2^*}\left(\frac{ \|\nabla u_\varepsilon\|_2^2}{\|u_\varepsilon\|_{2^*}^{2^*}}\right)^{\frac{2^*}{2^*-2}}\|u_\varepsilon\|_{2^*}^{2^*}+\frac{\max V}{2}\left(\frac{ \|\nabla u_\varepsilon\|_2^2}{\|u_\varepsilon\|_{2^*}^{2^*}}\right)^{\frac{2}{2^*-2}}\| u_\varepsilon\|_2^2    \\
&=& \frac{1}{N}\left(\frac{ \|\nabla u_\varepsilon\|_2^2}{\|u_\varepsilon\|_{2^*}^{2}}\right)^{\frac{2^*}{2^*-2}}+\frac{\max V}{2}\left(\frac{ \|\nabla u_\varepsilon\|_2^2}{\|u_\varepsilon\|_{2^*}^{2^*}}\right)^{\frac{2}{2^*-2}}\| u_\varepsilon\|_2^2  \\
&=& \frac{1}{N}\left[\frac{S^{\frac{N}{2}}+O\left(\varepsilon^{N-2}\right)}{(S^{\frac{N}{2}}+O(\varepsilon^{N-2}))^{\frac{2}{2^*}}}\right]^{\frac{2^*}{2^*-2}} +O\left(\varepsilon^{N-2}\right) \\
    &=& \frac{1}{N}S^{\frac{N}{2}}.
\end{eqnarray*}
We complete the proof.
\end{proof}

\begin{proof}[\bf Proof of Theorem \ref{t1.3}]
Recall that the sequence $\left\{u_n\right\} \subset E$ which satisfies \eqref{eq6.2}. Moreover, lemma \ref{L6.2} implies $\left\{u_n\right\}$ is bounded in $E$. Then there exists a  function $u \in E$ such that up to a subsequence, $u_n \rightharpoonup u$ in $E, u_n \rightarrow u$ in $L^s(\mathbb{R}^N), \forall s \in(2,2^*), u_n(x) \rightarrow u(x)$ a.e. in $\mathbb{R}^N$. Thus for any $\varphi \in C_0^{\infty}(\mathbb{R}^N)$, one has
$$
\begin{aligned}
0 & =\int_{\mathbb{R}^N} \nabla u_n \cdot \nabla \varphi \mathrm{d} x+\int_{\mathbb{R}^N} V(x)u_n \varphi \mathrm{d} x-\int_{\mathbb{R}^N}\left|u_n\right|^{2^*-2} u_n \varphi \mathrm{d} x+o(1) \\
& =\int_{\mathbb{R}^N} \nabla u \cdot \nabla \varphi \mathrm{d} x+\int_{\mathbb{R}^N} V(x)u  \varphi \mathrm{d} x-\int_{\mathbb{R}^N} u^{2^*-1} \varphi \mathrm{d} x .
\end{aligned}
$$
That is, $u$ is a solution of equation \eqref{eq1.99}. We claim that $u \neq 0$. Suppose by contradiction that $u=0$. Since $\left\{u_{n}\right\}$ is bounded in $E$, up to a subsequence we have that $\|\nabla u_{n}\|_2^2\rightarrow \ell \in \mathbb{R}$. Using  \eqref{eq6.2}, we have
\begin{eqnarray*}
  \langle I^{\prime}(u_{n}),u_{n}\rangle&=& \int_{\mathbb{R}^N}\left|\nabla u_{n}\right|^2 d x+  \int_{\mathbb{R}^N} V(x) u_{n}^2 d x-  \int_{\mathbb{R}^N}\left|u_{n}\right|^{2^*} d x \\
  &\rightarrow&0,
\end{eqnarray*}
hence
$$
 \|u_{n}\|_{2^*}^{2^*}=\|\nabla u_{n}\|_2^2 \rightarrow \ell
$$
as well. Therefore, $\ell \geq S \ell^{\frac{2}{2^*}}$, and we deduce that either $\ell=0$, or $\ell \geq S^{\frac{N}{2}}$. Let us suppose at first that $\ell \geq S^{\frac{N}{2}}$. Since $I(u_{n}) \rightarrow c<\frac{1}{N}S^{\frac{N}{2}}$, we have that
\begin{eqnarray*}
\frac{1}{N}S^{\frac{N}{2}}>c&\leftarrow&I(u_{n})+o(1) \\
& =&\frac{1}{2} \int_{\mathbb{R}^N}\left|\nabla u_{n}\right|^2 d x+\frac{1}{2} \int_{\mathbb{R}^N} V(x) u_{n}^2 d x-\frac{1}{2^*} \int_{\mathbb{R}^N}\left|u_{n}\right|^{2^*} d x  \\
& =&\frac{\ell}{N}\geq \frac{S^{\frac{N}{2}}}{N},
\end{eqnarray*}
which is not possible. If instead $\ell=0$, we have $\|u_{n}\|_{2^*}\rightarrow0$, $\|\nabla u_{n}\|_2\rightarrow0$ and
$F(u_{n})\rightarrow0$. But then $I(u_{n}) \rightarrow 0\neq c$, which gives again a contradiction. Thus, $u$ does not vanish identically.

Next, we show that $u$ can be chosen to be positive. From the definition of the functional $I$, obviously, $I(|u|)=I(u)$. Moreover, it is easy to see that
$$
c=I(u)=I(|u|) \geq c,
$$
which shows that $I_\mu(|u|)=c$. So we can replace $u$ by $|u|$. Moreover, if $u^*$ designates the Schwarz's Symmetrization of $u$, we know that
$$
\int_{\mathbb{R}^N}|\nabla u|^2 d x \geq \int_{\mathbb{R}^N}\left|\nabla u^*\right|^2 d x \text { and } \int_{\mathbb{R}^N}|u|^2 \mathrm{~d} x=\int_{\mathbb{R}^N}\left|u^*\right|^2 d x ,
$$
then $I\left(u^*\right)=c$, from where it follows that we can replace $u$ by $u^*$.

Now, we prove that $u(x)>0$ for all $x \in \mathbb{R}^N$. In fact, it is enough to apply the Harnack Inequality (see \cite[Theorem 8.20]{DGNS2001}). Assume by contradiction that there exists $x_0 \in \mathbb{R}^N$ such that $u\left(x_0\right)=0$. Since $u \neq 0$, there exists $x_1 \in \mathbb{R}^N$ such that $u\left(x_1\right)>0$. Have this in mind, fix $R>0$ large enough such that $x_0, x_1 \in B_R(0)$. By \cite[Theorem 8.20]{DGNS2001}, there exists $C>0$ such that
$$
\sup _{y \in B_R(0)} u(y) \leq C \inf _{y \in B_R(0)} u(y),
$$
which is absurd, because in this case
$$
\sup _{y \in B_R(0)} u(y)>0 \quad \text {and} \inf _{y \in B_R(0)} u(y)=0 .
$$
This completes the proof.
\end{proof}

\section{Some variational estimates}
From now on,  we assume that the positive ground state solution of \eqref{eq1.99} is $W(x)$ and $\nabla V(x) \cdot x\leq0$. The equation \eqref{eq1.99} gives
\begin{equation*}
  \int_{\mathbb{R}^N}(|\nabla W|^2+V(x)|W|^2)dx=\int_{\mathbb{R}^N}|W|^{2^*}dx.
\end{equation*}
Moreover, note that $W(x)$ is a solution of \eqref{eq1.99}, so we have the following Pohozaev identity
 \begin{equation*}
   \frac{N-2}{2N} \int_{\mathbb{R}^N}|\nabla W|^2 d x+\frac{1}{2 N} \int_{\mathbb{R}^N} (\nabla V(x) \cdot x) |W|^2dx+\frac{1}{2} \int_{\mathbb{R}^N} V(x) |W|^2 d x
   =\frac{1}{2^*} \int_{\mathbb{R}^N}\left|W\right|^{2^*} d x.
 \end{equation*}
Hence,
\begin{eqnarray*}
E(W)&=&\frac{1}{2} \int_{\mathbb{R}^N}(\left|\nabla W\right|^2+V(x)|W|^2) d x-   \frac{1}{2^*}\int_{\mathbb{R}^N}|W|^{2^*} d x \\
&=&\frac{1}{N}\int_{\mathbb{R}^N}|\nabla W|^2dx-\frac{1}{2 N} \int_{\mathbb{R}^N} (\nabla V(x) \cdot x) |W|^2dx .
\end{eqnarray*}
\begin{lemma}\label{L7.1}
Assume $u$  satisfies
$$
\|\nabla u\|_{2}^2<\|\nabla W\|_{2}^2 .
$$
Moreover, let $E(u)\leq(1-\delta_0)E(W)$, where $\delta_0>0$. Then, there exists $\delta_1>0,\ \bar{\delta}>0$ such that
\begin{equation}\label{eq7.1}
  \int_{\mathbb{R}^N}|\nabla u|^2dx \leq(1-\delta_1) \int_{\mathbb{R}^N}|\nabla W|^2dx,
\end{equation}
\begin{equation}\label{eq7.2}
  \int_{\mathbb{R}^N}(|\nabla u|^2-|u|^{2^*})dx \geq \bar{\delta} \int_{\mathbb{R}^N}|\nabla u|^2dx,
\end{equation}
\begin{equation}\label{eq7.3}
  E(u) \geq 0 .
\end{equation}
\end{lemma}
\begin{proof}
In order to get \eqref{eq7.2}, take $\|\nabla u\|_{2}^2$ very small, there exists $\bar{\delta}>0$ such that
\begin{eqnarray*}
  \int_{\mathbb{R}^N}(|\nabla u|^2-|u|^{2^*})dx &\geq& \|\nabla u\|_{2}^2 -C_{\varepsilon}\|\nabla u\|_{2}^{2^*}\\
    &=& \bar{\delta}\int_{\mathbb{R}^N}|\nabla u|^2dx,
\end{eqnarray*}
which implies \eqref{eq7.2} holds. To prove \eqref{eq7.1}, by $E(u)\leq(1-\delta_0)E(W)$ and $\|\nabla u\|_{2}^2<\|\nabla W\|_{2}^2$, we have
\begin{eqnarray*}
  &&\frac{1}{2} \int_{\mathbb{R}^N}(\left|\nabla u\right|^2+V(x)|u|^2) d x-   \frac{1}{2^*}\int_{\mathbb{R}^N}|u|^{2^*} d x \\
  &\leq &  \frac{1-\delta_0}{N}\int_{\mathbb{R}^N}|\nabla W|^2dx-\frac{1-\delta_0}{2 N} \int_{\mathbb{R}^N} (\nabla V(x) \cdot x) |W|^2dx,
\end{eqnarray*}
so there exists $\delta_1>0$ such that
\begin{eqnarray*}
  \int_{\mathbb{R}^N}|\nabla u|^2dx &\leq&\frac{1}{2} \int_{\mathbb{R}^N}(\left|\nabla u\right|^2-V(x)|u|^2) d x+ \frac{1-\delta_0}{N}\int_{\mathbb{R}^N}|\nabla W|^2dx+ \frac{1}{2^*}\int_{\mathbb{R}^N}|u|^{2^*} d x  \\
  &&-\frac{1-\delta_0}{2 N} \int_{\mathbb{R}^N} (\nabla V(x) \cdot x) |W|^2dx\\
  &\leq&\frac{1}{2}\int_{\mathbb{R}^N}|\nabla u|^2dx+ \frac{1-\delta_0}{N}\int_{\mathbb{R}^N}|\nabla W|^2dx+ C \|\nabla u\|_{2}^{2^*}  \\
  &&+\frac{1-\delta_0}{2 N}\|\nabla V(x) \cdot x\|_{\frac{N}{2}}S^{-1} \int_{\mathbb{R}^N} |\nabla W|^2dx\\
    &=&\frac{1}{2}\int_{\mathbb{R}^N}|\nabla W|^2dx+ \frac{1-\delta_0}{2N}\int_{\mathbb{R}^N}|\nabla W|^2dx  +C \|\nabla W\|_{2}^{2^*}\\
    &&+\frac{1-\delta_0}{2 N}\|\nabla V(x) \cdot x\|_{\frac{N}{2}}S^{-1} \int_{\mathbb{R}^N} |\nabla W|^2dx\\
    &:=&(1-\delta_1)\int_{\mathbb{R}^N}|\nabla W|^2dx,
\end{eqnarray*}
where we use the fact that the negative part of $V(x)$ is small and $\nabla V(x) \cdot x\in L^{\frac{N}{2}}(\mathbb{R}^N)$. Finally, we show that \eqref{eq7.3}  holds. Indeed,
\begin{eqnarray*}
  E(u) &=& \frac{1}{2} \int_{\mathbb{R}^N}(\left|\nabla u\right|^2+V(x)|u|^2) d x-   \frac{1}{2^*}\int_{\mathbb{R}^N}|u|^{2^*} d x  \\
    &\geq&\frac{1}{2}\left(1-\left\|V_{-}\right\|_{\frac{N}{2}} S^{-1}\right) \int_{\mathbb{R}^N}|\nabla u|^2 d x-\frac{1}{2^*\cdot S^{
\frac{2^*}{2}}}\left(\int_{\mathbb{R}^N}|\nabla u|^2 d x\right)^{\frac{2^*}{2}}\\
&\geq&0
\end{eqnarray*}
since $\left\|V_{-}\right\|_{\frac{N}{2}}\leq S$, this completes the proof.
\end{proof}
\begin{remark}\label{r7.1}
From Lemma \ref{L7.1}, we know that the selection of $W$ is not arbitrary. In fact, we need to choose $W$ such that $\int_{\mathbb{R}^N}|\nabla W|^2dx$ is small. Moreover, due to the nonlinear term is odd, there are actually infinite standing wave solutions for equation \eqref{eq1.1}, see \cite{HBII1983}. In this case, we only need to take the solution that minimizes $\int_{\mathbb{R}^N}|\nabla W|^2dx$.
\end{remark}
\begin{corollary}\label{c5.1}
Assume that $u\in H^1$ and $\int_{\mathbb{R}^N}\left|\nabla u\right|^2dx<\int_{\mathbb{R}^N}|\nabla W|^2dx$, then $E(u)\geq0$.
\end{corollary}
\begin{proof}
If $E(u)\geq E(W)=\frac{1}{N}\int_{\mathbb{R}^N}|\nabla W|^2dx-\frac{1}{2 N} \int_{\mathbb{R}^N} (\nabla V(x) \cdot x) |W|^2dx $, this is obvious. If $E(u)< E(W)$, the claim follows from \eqref{eq7.3}.
\end{proof}

\begin{theorem}\label{t5.1}(Energy trapping). Let $u$ be a solution of the \eqref{eq1.1}, with $t_0=0,\left.u\right|_{t=0}=u_0$ such that for $\delta_0>0$,
$$
\int_{\mathbb{R}^N}\left|\nabla u_0\right|^2dx<\int_{\mathbb{R}^N}|\nabla W|^2dx,\ E(u_0)<(1-\delta_0)E(W).
$$
Let $I(0\in I)$ be the maximal interval of existence. Let $\delta_0,\ \bar{\delta}$ be as in Lemma \ref{L7.1}. Then, for each $t \in I$, we have
\begin{equation}\label{eq7.4}
  \int_{\mathbb{R}^N}|\nabla u(t)|^2dx \leq(1-\delta_1) \int_{\mathbb{R}^N}|\nabla W|^2dx,
\end{equation}
\begin{equation}\label{eq7.5}
  \int_{\mathbb{R}^N}(|\nabla u(t)|^2-|u(t)|^{2^*})dx \geq  \bar{\delta} \int_{\mathbb{R}^N}|\nabla u(t)|^2dx,
\end{equation}
\begin{equation}\label{eq7.6}
  E(u(t))  \geq  0 .
\end{equation}
\end{theorem}
\begin{proof} By energy conservation, $E(u(t))=E(u_0),\ t \in I$ and the theorem follows directly from Lemma \ref{L7.1} and a continuity argument.
\end{proof}
\begin{corollary}\label{c5.2}
Let $u, u_0$ be as in Theorem \ref{t5.1}. Then for all $t \in I$ we have $E(u(t)) \simeq \int_{\mathbb{R}^N}|\nabla u(t)|^2dx \simeq \int_{\mathbb{R}^N}\left|\nabla u_0\right|^2dx$, with comparability constants which depend only on $N$.
\end{corollary}
\begin{proof}
Note that
\begin{eqnarray*}
E(u(t)) &=&\frac{1}{2} \int_{\mathbb{R}^N}(\left|\nabla u(t)\right|^2+V(x)|u(t)|^2) d x-   \frac{1}{2^*}\int_{\mathbb{R}^N}|u(t)|^{2^*} d x\\
&\leq& \frac{1}{2}\left(1+\left\|V \right\|_{\frac{N}{2}} S^{-1}\right) \int_{\mathbb{R}^N}|\nabla u(t)|^2 d x-    \frac{1}{2^*}\int_{\mathbb{R}^N}|u(t)|^{2^*} d x\\
&\leq& \frac{1}{2}\left(1+\left\|V \right\|_{\frac{N}{2}} S^{-1}\right) \int_{\mathbb{R}^N}|\nabla u(t)|^2 d x,
\end{eqnarray*}
and by the proof of \eqref{eq7.3} there exists $\widetilde{\delta}$ such that
\begin{eqnarray*}
E(u(t)) &\geq&\frac{1}{2}\left(1-\left\|V_{-}\right\|_{\frac{N}{2}} S^{-1}\right) \int_{\mathbb{R}^N}|\nabla u(t)|^2 d x-\frac{1}{2^*\cdot S^{
\frac{2^*}{2}}}\left(\int_{\mathbb{R}^N}|\nabla u(t)|^2 d x\right)^{\frac{2^*}{2}}\\
&\geq& \widetilde{\delta} \int_{\mathbb{R}^N}|\nabla u(t)|^2dx,
\end{eqnarray*}
so the first equivalence follows. For the second one note that
\begin{equation*}
  E(u(t))= E(u_0) \simeq \int_{\mathbb{R}^N}\left|\nabla u_0\right|^2dx,
\end{equation*}
by the first equivalence when $t=0$.
\end{proof}

\section{Concentration compactness}

The main purpose of this section is to prove the following concentration compactness lemma, which plays a crucial role in the proof of the theorem in the next section.

\begin{lemma}\label{L8.1}(Concentration compactness)
Let $\left\{v_{0, n}\right\} \in H^1, \left\|v_{0, n}\right\|_{H^1}$ $\leq\rho$ . Assume that $\left\|e^{-i t \mathcal{L}} v_{0, n}\right\|_{L_{t,x}^{\frac{2(N+2)}{N-2}}} \geq \rho>0$, where $\rho$ is as in Lemma \ref{L3.4}. Then there exists a sequence $\left\{V_{0, j}\right\}_{j=1}^{\infty}$ in $H^1$, a subsequence of $\left\{v_{0, n}\right\}$$($which we still call $\left\{v_{0, n}\right\})$  and the parameters $\left(  x_{j, n} ; t_{j, n}\right) \in   \mathbb{R}^N \times \mathbb{R}$, with
$$
 \left|t_{j, n}-t_{j^{\prime}, n}\right| + \left|x_{j, n}-x_{j^{\prime}, n}\right| \rightarrow \infty
$$
as $n \rightarrow \infty$ for $j \neq j^{\prime}$$($we say that $(  x_{j, n} ; t_{j, n})$ is orthogonal if this property is verified$)$ such that
\begin{equation}\label{eq8.1}
  \left\|V_{0,1}\right\|_{\dot{H}^1} \geq \alpha_0(\rho)>0 .
\end{equation}

If $V_j^l(x, t)=e^{-i t \mathcal{L}} V_{0, j}$, then, given $\varepsilon_0>0$, there exists $J=J\left(\varepsilon_0\right)$ and
\begin{equation}\label{eq8.2}
  \left\{w_n\right\}_{n=1}^{\infty} \in H^1\ \text{so that}\  v_{0, n}=\sum_{j=1}^J  V_j^l(x-x_{j, n}, -t_{j, n})+w_n
\end{equation}
with $\left\|e^{-i t \mathcal{L}} w_n\right\|_{L_{(-\infty, +\infty)}^{\frac{2(N+2)}{N-2}}W^{1,{\frac{2N(N+2)}{N^2+4}}}} \leq \varepsilon_0$, for n large
\begin{equation}\label{eq8.3}
 \|(\mathcal{L})^{\frac{s}{2}} v_{0, n}\|_2^2 =\sum_{j=1}^J \|(\mathcal{L})^{\frac{s}{2}}e^{it_{j, n}\mathcal{L}^{j,n}} V_{0, j}\|_2^2 + \|(\mathcal{L})^{\frac{s}{2}} w_n\|^2+o(1),\ n \rightarrow \infty,\ s\in\{0,1\},
\end{equation}
\begin{equation}\label{eq8.4}
   E\left(v_{0, n}\right)=\sum_{j=1}^J E (V_j^l ( -t_{j, n}  ) )+E\left(w_n\right)+o(1),\ n \rightarrow \infty.
\end{equation}
\end{lemma}
The proof of this lemma originates from Keraani \cite{KSK2001}, but we need to modify the proof since this paper considers the different operators. Firstly, we consider linear equation
\begin{equation}\label{eq8.5}
  \left\{\begin{array}{l}
i \partial_t u+\Delta u -V(x)u =0,\ (x, t) \in \mathbb{R}^N \times \mathbb{R}, \\
\left.u(0,x)\right|_{t=0}=\varphi(x) \in  H ^1(\mathbb{R}^N).
\end{array}\right.
\end{equation}

\begin{lemma}\label{L8.2}
Let $(\varphi_n)_{n \geq  0}$ be a bounded sequence in $H^1(\mathbb{R}^N)$. Let $(v_n)_{n \geq  0}$ be the sequence of solutions to \eqref{eq8.5} with initial data $v_n(x, 0)=\varphi_n(x)$. Then there exist a subsequence $(v_n^{\prime})$ of $(v_n)$, a sequence $(\mathbf{1}^j)_{j \geq  1}$ of scales, a sequence $(\mathbf{z}^j)_{j \geq 1}$ of cores and a sequence $( e^{-it\mathcal{L}}V^j)_{j \geq  1}\subset H^1(\mathbb{R}^N)$, such that

$\mathrm{(i)}$ the pairs $\left(\mathbf{1}^j, \mathbf{z}^j\right)$ are pairwise orthogonal;

$\mathrm{(ii)}$ for every $l \geq 1$,
 \begin{equation}\label{eq8.6}
  v_n^{\prime}(x, t)=\sum_{j=1}^l   e^{-it\mathcal{L}}V ^j(x, t)+w_n^l(x, t),
 \end{equation}
with
 \begin{equation}\label{eq8.7}
   \limsup\limits_{n \rightarrow \infty}\left\|e^{-it\mathcal{L}}w_n^l\right\|_{L^q\left(\mathbb{R}, L^r\left(\mathbb{R}^N\right)\right)} \rightarrow 0, \ l \rightarrow \infty
 \end{equation}
for every $H^1$-admissible pair $(q, r)$(defined in Definition \ref{D2.3}), and, for every $l \geq 1$,
 \begin{equation}\label{eq8.8}
   \|(\mathcal{L})^{\frac{s}{2}}v_{ n}'\|_2^2 =\sum_{j=1}^l \|(\mathcal{L})^{\frac{s}{2}} e^{-it\mathcal{L}}V^j \|_2^2 + \|(\mathcal{L})^{\frac{s}{2}} w_n^l\|_2^2+o(1),\ n \rightarrow \infty,\ s\in\{0,1\}.
 \end{equation}
\end{lemma}
\begin{proof}
The proof is divided into three steps.

\textbf{Step 1.} This first step is devoted to the determination of the family of
scales. We recall some general results of decomposition of bounded
sequences in $L^2(\mathbb{R}^N)$, see \cite{HBPG1999}.

\begin{definition}\label{D8.1}
Let $\mathbf{f}=\left(f_n\right)_{n \geqslant 0}$ be a bounded sequence of $L^2(\mathbb{R}^N)$ and $\mathbf{h}=\left(h_n\right)_{n \geqslant 0}$ a scale.
$\mathrm{(i)}$ We say that $\mathbf{f}$ is $\mathbf{h}$-oscillatory if
$$
\limsup _{n \rightarrow \infty}\left(\int_{h_n|\xi| \leq\frac{1}{R}}|\hat{f}_n(\xi)|^2 d \xi+\int_{h_n|\xi| \geq R}|\hat{f}_n(\xi)|^2 d \xi\right)\rightarrow 0 ,\ R \rightarrow \infty.
$$
$\mathrm{(ii)}$ We say that $\mathbf{f}$ is $\mathbf{h}$-singular if, for every $b>a>0$, we have
$$
\int_{a \leq  h_n|\xi| \leq  b}|\hat{f}_n(\xi)|^2 d \xi \rightarrow 0,\ n \rightarrow \infty.
$$
\end{definition}
\begin{definition}\label{D8.2}
We say that two scales $\mathbf{h}=(h_n)$ and $\tilde{\mathbf{h}}=(\tilde{h}_n)$ are orthogonal$($we note $\mathbf{h} \perp \tilde{\mathbf{h}})$ if
$$
\frac{h_n}{\tilde{h}_n}+\frac{\tilde{h}_n}{h_n} \rightarrow+\infty ,\ n \rightarrow \infty.
$$
\end{definition}
The following remark will be useful.
\begin{remark}\label{R8.1}
$\mathrm{(i)}$ Let $\mathbf{h}=\left(h_n\right)_{n \geq 0}$ be a scale. Let $\mathbf{f}$ and $\mathbf{g}$ be two bounded sequences in $L^2(\mathbb{R}^N)$, such that $\mathbf{f}$ is $\mathbf{h}$-oscillatory and $\mathbf{g}$ is h-singular. Then, via Plancherel's inversion formula and Cauchy-Schwartz inequality, they are decoupled in infinity, in the sense
\begin{equation}\label{eq8.9}
 \int_{\mathbb{R}^N} f_n(x) \bar{g}_n(x) d x\rightarrow 0,\ n \rightarrow \infty.
\end{equation}
From \eqref{eq8.9}, it follows that
$$
\left\|f_n+g_n\right\|_{L^2(\mathbb{R}^N)}^2=\left\|g_n\right\|_{L^2(\mathbb{R}^N)}^2+\left\|f_n\right\|_{L^2(\mathbb{R}^N)}^2+o(1), \ n \rightarrow+\infty .
$$
$\mathrm{(ii)}$ Let $\mathbf{h}=(h_n)_{n \geq0}$ be a scale and $\mathbf{f}$ a bounded sequence in $L^2(\mathbb{R}^N)$, such that $\mathbf{f}$ is $\mathbf{h}$-oscillatory. Then $\mathbf{f}$ is $\tilde{\mathbf{h}}$-singular for every scale $\tilde{\mathbf{h}}$ orthogonal to $\mathbf{h}$.
\end{remark}
\begin{proposition}\label{P8.1}
Let $\mathbf{f}$ be a bounded sequence in $L^2(\mathbb{R}^N)$. Then there exist a subsequence $\mathbf{f}^{\prime}$ of $\mathbf{f}$, a family $\left(\mathbf{h}^j\right)_{j \geq  1}$ of pairwise orthogonal scales and a family $\left(\mathbf{g}^j\right)_{j \geq  1}$ of bounded sequences in $L^2(\mathbb{R}^N)$, such that\\
$\mathrm{(i)}$ for every $j, \mathbf{g}^j$ is $\mathbf{h}^j$-oscillatory;\\
$\mathrm{(ii)}$ for every $l \geq  1$ and $x \in \mathbb{R}^N$,
$$
f_n^{\prime}(x)=\sum_{j=1}^l g_n^j(x)+R_n^l(x),
$$
where $(R_n^l)$ is $\mathbf{h}^j$-singular for every $j \in\{1, \ldots, l\}$, and
$$
\underset{n \rightarrow \infty}{\limsup }\left\|R_n^l\right\|_{\dot{B}_{2, \infty}^0}\rightarrow0,\ l \rightarrow \infty ;
$$
$\mathrm{(iii)}$ for every $l \geqslant 1$,
$$
\|f_n^{\prime}\|_{L^2(\mathbb{R}^N)}^2=\sum_{j=1}^l\|g_n^j\|_{L^2(\mathbb{R}^N)}^2+\|R_n^l\|_{L^2(\mathbb{R}^N)}^2+o(1), \  n \rightarrow+\infty.
$$
\end{proposition}

Applying Proposition \ref{P8.1} to the sequence $\left((\mathcal{L})^{\frac{s}{2}}\varphi_n\right)_{n \geq  0},\ s\in\{0,1\}$, we obtain a family of scales $\left(\mathbf{h}^j\right)_{j \geq 1}$ and a family $\left(\varphi^j\right)_{j \geq 1}$ of bounded sequences in $H^1(\mathbb{R}^N)$ such that
\begin{equation}\label{eq8.10}
  \varphi_n^{\prime}(x)=\sum_{j=1}^l \varphi_n^j(x)+\Phi_n^l(x),
\end{equation}
where $((\mathcal{L})^{\frac{s}{2}}\varphi_n^j)$ is $\mathbf{h}^j$-oscillatory for every $j \geq 1$, $((\mathcal{L})^{\frac{s}{2}} \Phi_n^l)$ is $\mathbf{h}^j$-singular for every $j \in\{1,2, \ldots, l\}$, and
 \begin{equation}\label{eq-8.11}
   \underset{n \rightarrow \infty}{\limsup }\left\|(\mathcal{L})^{\frac{s}{2}} \Phi_n^l\right\|_{\dot{B}_{2, \infty}^0} \rightarrow0 ,\ l \rightarrow \infty.
 \end{equation}
Furthermore, the following almost orthogonality identity holds
\begin{equation}\label{eq8.11}
  \|(\mathcal{L})^{\frac{s}{2}}\varphi_n^{\prime}\|_2^2 =\sum_{j=1}^l \|(\mathcal{L})^{\frac{s}{2}}\varphi_n^j \|_2^2 + \|(\mathcal{L})^{\frac{s}{2}}\Phi_n^l\|_2^2+o(1),\ n \rightarrow \infty
\end{equation}
for every $l \geq 1$. To \eqref{eq8.10} corresponds a decomposition of $\left(v_n^{\prime}\right)$ solutions of \eqref{eq8.5}
 \begin{equation}\label{eq8.99}
   v_n^{\prime}(x, t)=\sum_{j=1}^l p_n^j(x, t)+q_n^l(x, t).
 \end{equation}
Note that, we have the conservation law for \eqref{eq8.5}, that is
\begin{equation}\label{eq8.13}
    \int_{\mathbb{R}^N}(\left|\nabla u(x, t)\right|^2+V(x)|u(x,t)|^2) d x= \int_{\mathbb{R}^N}(\left|\nabla u_0(x)\right|^2+V(x)|u_0(x)|^2) d x.
\end{equation}
From \eqref{eq8.11} and \eqref{eq8.13}, we infer
\begin{equation}\label{eq8.14}
  \|(\mathcal{L})^{\frac{s}{2}} v_n^{\prime}\|_2^2 =\sum_{j=1}^l \|(\mathcal{L})^{\frac{s}{2}} p_n^j \|_2^2 + \|(\mathcal{L})^{\frac{s}{2}} q_n^l\|_2^2+o(1),\ n \rightarrow \infty
\end{equation}
for every $l \geq 1$.

To estimate the remainder term $q_n^l$, we need the following refined Sobolev inequality.
\begin{proposition}\label{P8.2}
For every $f \in H^1(\mathbb{R}^N)$, we have
$$
\|f\|_{L^{\frac{2N}{N-2}}(\mathbb{R}^N)} \leq  C\|(\mathcal{L})^{\frac{s}{2}} f\|_{L^2(\mathbb{R}^N)}^{1-\frac{2}{N}}\|(\mathcal{L})^{\frac{s}{2}} f\|_{\dot{B}_{2, \infty}^0}^{\frac{2}{N}} .
$$
\end{proposition}
\begin{proof}
By using the Proposition 1.41 and Theorem 1.43 in \cite{HB2011}, we know
$$
\|f\|_{L^{\frac{2N}{N-2}}(\mathbb{R}^N)} \leq  C\|\nabla f\|_{L^2(\mathbb{R}^N)}^{1-\frac{2}{N}}\|\nabla f\|_{\dot{B}_{2, \infty}^0}^{\frac{2}{N}} .
$$
Hence, it follows from Lemma \ref{L3.3} that
$$
\|f\|_{L^{\frac{2N}{N-2}}(\mathbb{R}^N)} \leq  C\|(\mathcal{L})^{\frac{s}{2}} f\|_{L^2(\mathbb{R}^N)}^{1-\frac{2}{N}}\|(\mathcal{L})^{\frac{s}{2}} f\|_{\dot{B}_{2, \infty}^0}^{\frac{2}{N}},
$$
which is the result we require.
\end{proof}
If $q$ is a solution of \eqref{eq8.5} then $\sigma_k(D) q$ is also solution to the same equation, where $\sigma_k(\xi)=\mathbf{1}_{\left\{2^k \leq |\xi| \leq  2^{k+1}\right\}}(\xi)$. The conservation law for all $\sigma_k(D) q$, $k \in \mathbb{Z}$, implies
\begin{equation}\label{eq8.16}
  \|(\mathcal{L})^{\frac{s}{2}}q(t)\|_{\dot{B}_{2, \infty}^0}=\|(\mathcal{L})^{\frac{s}{2}} q(0)\|_{\dot{B}_{2, \infty}^0} .
\end{equation}
Applying \eqref{eq8.16} to $q_n^l$, we obtain
\begin{equation}\label{eq8.18}
 \left\|(\mathcal{L})^{\frac{s}{2}} q_n^l\right\|_{L^{\infty}\left(\mathbb{R}, \dot{B}_{2, \infty}^0\right)}=\left\|(\mathcal{L})^{\frac{s}{2}} \Phi_n^l\right\|_{\dot{B}_{2, \infty}^0} .
\end{equation}
Using \eqref{eq-8.11} and \eqref{eq8.18}, we have
\begin{equation*}
\underset{n \rightarrow \infty}{\limsup }\left\|(\mathcal{L})^{\frac{s}{2}} q_n^l\right\|_{L^{\infty}\left(\mathbb{R}, \dot{B}_{2, \infty}^0\right)}\rightarrow0, \ l \rightarrow \infty.
\end{equation*}
According to the Proposition \ref{P8.2},
\begin{equation}\label{eq8.19}
\limsup _{n \rightarrow \infty}\left\|q_n^l\right\|_{L^{\infty}(\mathbb{R}, L^{\frac{2N}{N-2}})} \leq  C \limsup _{n \rightarrow \infty} \|(\mathcal{L})^{\frac{s}{2}} q_n^l\|_2^{\frac{1}{3}} \limsup _{n \rightarrow \infty}\left\|(\mathcal{L})^{\frac{s}{2}} q_n^l\right\|_{L^{\infty}(\mathbb{R}, \dot{B}_{2, \infty}^0)}^{\frac{2}{3}} .
\end{equation}
Moreover, it follows from \eqref{eq8.14} that
\begin{equation}\label{eq8.20}
\limsup _{n \rightarrow \infty} \|(\mathcal{L})^{\frac{s}{2}} q_n^l\|_2^2 \leq  \limsup _{n \rightarrow \infty} \|(\mathcal{L})^{\frac{s}{2}} v_n^{\prime}\|_2^2 \leq \limsup _{n \rightarrow \infty} \|(\mathcal{L})^{\frac{s}{2}} \varphi_n^{\prime}\|_2^2 \leq  C
\end{equation}
for every $l \geq  1$. Combining \eqref{eq8.19} and \eqref{eq8.20}, it holds
$$
\limsup _{n \rightarrow \infty}\left\|q_n^l\right\|_{L^{\infty}(\mathbb{R}, L^{\frac{2N}{N-2}})} \rightarrow0 ,\ l \rightarrow \infty.
$$
Let $(s, r)$ be a $H^1$-admissible pair. Due to interpolation inequality, we know that
$$
\|q_n^l\|_{L^s(\mathbb{R}, L^r)} \leq\|q_n^l\|_{L^{\infty}(\mathbb{R}, L^{\frac{2N}{N-2}})}^\alpha\|q_n^l\|_{L^{\tilde{s}}(\mathbb{R}, L^{\tilde{r}})}^\beta,
$$
where
$$
\tilde{s}=\frac{s(r-N)}{r}, \quad \tilde{r}=\frac{2(r-N)}{4-N}, \quad \beta= 1-\frac{N}{r} , \ \alpha=\frac{N}{r} .
$$
It is easy to check that the pair $(\tilde{s}, \tilde{r})$ is $H^1$-admissible. Hence, by Proposition \ref{P2.2}, we have
$$
\limsup _{n \rightarrow \infty}\|q_n^l\|_{L^{\tilde{s}}(\mathbb{R}, L^{\tilde{r}})} \leq  C \limsup _{n \rightarrow \infty} \|(\mathcal{L})^{\frac{s}{2}} q_n^l\|_{2} \leq  C \limsup _{n \rightarrow \infty} \|(\mathcal{L})^{\frac{s}{2}} v_n'\|_{2} \leq C .
$$
Therefore, it follows that
$$
\limsup _{n \rightarrow \infty}\left\|q_n^l\right\|_{L^s\left(\mathbb{R}, L^{r}\right)} \rightarrow 0,\ l \rightarrow \infty
$$
for every $H^1$-admissible pair $(s, r)$.

\textbf{Step 2.} This step is devoted to the determination of the families of cores $\left(\mathbf{z}^j\right)_{j \geq  1}$ and profiles $\left(V^j\right)_{j \geq 1}$. We denote by $\mathbf{1}$ the scale with all terms equal to 1. Our main tool is the following
\begin{proposition}\label{P8.3}
Assume that $\mathbf{P}=(P_n)_{n \geq  0}$ be a sequence of solutions to \eqref{eq8.5} such that $((\mathcal{L})^{\frac{s}{2}} P_n(\cdot, 0))_{n \geq  0}$ is bounded in $L^2(\mathbb{R}^N)$ and $\mathbf{1}$-oscillatory. Then there exist a subsequence $(P_n^{\prime})$ of $(P_n)$, a family $(\mathbf{z}^\alpha)_{\alpha \geq 1}=(\mathbf{x}^\alpha, \mathbf{t}^\alpha)_{\alpha \geq 1} \subset \mathbb{R}^N \times \mathbb{R}$ of cores and a family $(e^{-it\mathcal{L}}V^\alpha)_{\alpha \geq 1}$ of solutions to \eqref{eq8.5}, such that

$\mathrm{(i)}$ for every $\alpha \neq \beta,\ \left|z_n^\alpha-z_n^\beta\right|\rightarrow+\infty$ as $n \rightarrow \infty$;

$\mathrm{(ii)}$ for every $A \geq 1$, every $x \in \mathbb{R}^N$ and $t \in \mathbb{R}$, we have
$$
P_n^{\prime}(t, x)=\sum_{\alpha=1}^A e^{-i(t-t_n^\alpha)\mathcal{L}^n}V_n^\alpha(x-x_n^\alpha, t-t_n^\alpha)+P_n^A(t, x),
$$
where
\begin{equation}\label{eq8.21}
  \limsup _{n \rightarrow \infty}\|P_n^A\|_{L^s(\mathbb{R}, L^r(\mathbb{R}^N))}\rightarrow 0,\ A \rightarrow \infty
\end{equation}
for every $H^1$-admissible pair $(s, r)$, and
\begin{equation}\label{eq8.22}
\|(\mathcal{L})^{\frac{s}{2}} P_n^{\prime}\|_2^2=\sum_{\alpha=1}^A \|(\mathcal{L})^{\frac{s}{2}} e^{-i(t-t_n^\alpha)\mathcal{L}^n}V_n^\alpha\|_2^2+\|(\mathcal{L})^{\frac{s}{2}} P_n^A\|_2^2+o(1), \ n \rightarrow \infty.
\end{equation}
\end{proposition}
\begin{proof}
Assume that $\mathcal{V}(\mathbf{P})$ be the set of solutions to \eqref{eq8.5} obtained as weak limits in $L^{\infty}(\mathbb{R}, H^1(\mathbb{R}^N))$ of subsequences of translated $(P_n(\cdot+y_n, \cdot+t_n))$ of $\mathbf{P}$. Define
$$
\eta(\mathbf{P}):=\sup \left\{\|(\mathcal{L})^{\frac{s}{2}} Q\|_2: Q \in \mathcal{V}(\mathbf{P})\right\}.
$$
Obviously,
$$
\eta(\mathbf{P}) \leq  \limsup _{n \rightarrow \infty} \|(\mathcal{L})^{\frac{s}{2}} P_n\|_2.
$$
Next, we will show that for every sequence $\mathbf{P}$ there exist a sequence $(e^{-it\mathcal{L}}V^\alpha)_{\alpha \geq  1}$ of $\mathcal{V}(\mathbf{P})$ and a family $(y_n^\alpha, t_n^\alpha) \subset \mathbb{R}^N \times \mathbb{R}$, such that
\begin{equation}\label{eq8.23}
\alpha \neq \beta \Rightarrow \left|y_n^\alpha-y_n^\beta\right|+\left|t_n^\alpha-t_n^\beta\right|  \rightarrow \infty,\ n \rightarrow \infty
\end{equation}
and, going if necessary to a subsequence, the sequence $\left(P_n\right)_{n \geq 0}$ can be written as
\begin{equation}\label{eq8.24}
P_n(y,t)=\sum_{\alpha=1}^A  e^{-i(t-t_n^\alpha)\mathcal{L}^n}V_n^\alpha(y-y_n^\alpha, t-t_n^\alpha)+P_n^A(t, x), \ \eta(\mathbf{P}^A)\rightarrow0,\ A \rightarrow \infty
\end{equation}
with the almost orthogonality identity
\begin{equation}\label{eq8.25}
\|(\mathcal{L})^{\frac{s}{2}} P_n\|_2^2=\sum_{\alpha=1}^A \|(\mathcal{L})^{\frac{s}{2}} e^{i(t-t_n^\alpha)\mathcal{L}^n}V_n^\alpha\|_2^2+\|(\mathcal{L})^{\frac{s}{2}} P_n^A\|_2^2+o(1), \ n \rightarrow \infty .
\end{equation}
In fact, if $\eta(\mathbf{P})=0$, we can take $e^{-it\mathcal{L}}V^\alpha \equiv 0$ for all $\alpha$, otherwise we choose $e^{-it\mathcal{L}}V^1 \in \mathcal{V}(\mathbf{P})$, such that
$$
\|(\mathcal{L})^{\frac{s}{2}} e^{-it\mathcal{L}}V^1\|_2 \geq  \frac{1}{2} \eta(\mathbf{P})>0 .
$$
According to the definition, going if necessary to a subsequence, there exists some sequence $(y_n^1, t_n^1)$ of $\mathbb{R}^N \times \mathbb{R}$ such that
$$
P_n(\cdot+y_n^1, \cdot+t_n^1) \rightharpoonup e^{-it\mathcal{L}}V^1(y, t) .
$$
Let
\begin{equation}\label{eq8.26}
P_n^1(y, t)=P_n(y, t)-e^{-i(t-t_n^1)\mathcal{L}^n}V^1(y-y_n^1, t-t_n^1) .
\end{equation}
It follows from Lemma \ref{L2.5} that
\begin{equation*}
  P_n^1(\cdot+t_n^1, \cdot+y_n^1) \rightharpoonup 0,
\end{equation*}
which implies
$$
\|(\mathcal{L})^{\frac{s}{2}} P_n\|_2^2=\|(\mathcal{L})^{\frac{s}{2}} e^{-i(t-t_n^1)\mathcal{L}^n}V^1\|_2^2+\|(\mathcal{L})^{\frac{s}{2}} P_n^1\|_2^2+o(1), \  n \rightarrow \infty .
$$

Next, we replace $\mathbf{P}$ by $\mathbf{P}^1$ and repeat the same process. If $\eta(\mathbf{P}^1)>0$ we obtain $V^2$, $(y_n^2, t_n^2)$ and $\mathbf{P}^2$. Moreover, we have
$$
\left|y_n^1-y_n^2\right|+\left|t_n^1-t_n^2\right|  \rightarrow \infty, \ n \rightarrow \infty,
$$
otherwise, going if necessary to a subsequence, $P_n^1(\cdot+y_n^2, \cdot+t_n^2)\rightharpoonup0$, which implies that $e^{-it\mathcal{L}}V^2=0$, so $\eta(\mathbf{P}^1)=0$, this is impossible. An argument of iteration and orthogonal extraction allows us to construct the family $(y_n^\alpha, s_n^\alpha)$ and $(V^\alpha)_{\alpha \geq1}$ satisfying the claims \eqref{eq8.23} and \eqref{eq8.25}. Moreover, the convergence of the series $\sum\limits_{\alpha \geq 1} \|(\mathcal{L})^{\frac{s}{2}} e^{-it\mathcal{L}}V^\alpha\|_2^2$ implies that
$$
\|(\mathcal{L})^{\frac{s}{2}} e^{-it\mathcal{L}}V^\alpha\|_2^2 \rightarrow0 ,\ \alpha \rightarrow \infty.
$$
However, by construction, we have
$$
\eta(\mathbf{P}^\alpha)^2 \leq  \|(\mathcal{L})^{\frac{s}{2}} e^{-it\mathcal{L}}V^{\alpha-1}\|_2^2,
$$
which proves \eqref{eq8.24}.

To complete the proof of Proposition \ref{P8.3}, we need to prove \eqref{eq8.21}. This is the subject of the next paragraph. First, let us remark that if we apply the operator $\sigma_R(D)$ to both sides of \eqref{eq8.26}, where $\sigma_R=$ $\mathbf{1}_{\{|  \xi | \leq  \frac{1}{R}\} \cup\{|\xi| \geq  R\}}, R>0$, then we know that
$$
\|(\mathcal{L})^{\frac{s}{2}} \sigma_R(D) P_n\|_2^2=\|(\mathcal{L})^{\frac{s}{2}}\sigma_R(D) e^{-it\mathcal{L}}V^1\|_2^2+\|(\mathcal{L})^{\frac{s}{2}}\sigma_R(D) P_n^1\|_2^2+o(1), \  n \rightarrow \infty .
$$
By iteration, we infer
\begin{equation}\label{eq8.27}
\|(\mathcal{L})^{\frac{s}{2}}\sigma_R(D) P_n\|_2^2=\sum^A \|(\mathcal{L})^{\frac{s}{2}}\sigma_R(D) e^{-it\mathcal{L}}V^\alpha\|_2^2+\|(\mathcal{L})^{\frac{s}{2}}\sigma_R(D) P_n^A\|_2^2+o(1), \ n \rightarrow \infty .
\end{equation}
By \eqref{eq8.27} and Lemma \ref{L3.3}, choose $s=1$(similarly, it can be concluded that the case of $s=0$) we have
\begin{equation*}
  \limsup _{n \rightarrow \infty} \int_{\{|\xi| \leq \frac{1}{R}\} \cup\{|\xi| \geq  R\}}|\xi|^2|\hat{P}_n^A(\xi, 0)|^2 d \xi
 \leq  \limsup _{n \rightarrow \infty} \int_{\{|\xi| \leq  \frac{1}{R}\} \cup\{|\xi| \geq R\}}|\xi|^2|\hat{P}_n(\xi, 0)|^2 d \xi
\end{equation*}
for every $A \geq 1$ and very $R>0$. In particular, $(\nabla \mathbf{P}^A)$ is $\mathbf{1}$-oscillatory, for every $A \geq1$ (remember that, by hypothesis, $(\nabla P_n(\cdot, 0))_{n \geq 0}$ is $\mathbf{1}$-oscillatory). Let us now summarize the properties of the family $(\mathbf{P}^A)_{A \geq  1}$.

$\mathrm{(i)}$ For every $A \geq 1, \mathbf{P}^A$ is uniformly (on $n$ and $A$ ) bounded energy solutions to \eqref{eq8.5}.

$\mathrm{(ii)}$ For every $A \geq  1$ and every $R>0$,
\begin{equation}\label{eq8.28}
   \limsup _{n \rightarrow \infty} \int_{\{|\xi| \leq\frac{1}{R}\} \cup\{|\xi| \geq R\}}|\xi|^2|\hat{P}_n^A(\xi, 0)|^2 d \xi \\
   \leq  \limsup _{n \rightarrow \infty} \int_{\{|\xi| \leq\frac{1}{R}\} \cup\{|\xi| \geq R\}}|\xi|^2|\hat{P}_n(\xi, 0)|^2 d \xi .
\end{equation}

$\mathrm{(iii)}$
$$
\eta(\mathbf{P}^A)\rightarrow 0,\ A \rightarrow \infty .
$$

Using these properties, we shall prove that
$$
\limsup _{n \rightarrow \infty}\|P_n^A\|_{L^{\infty}(\mathbb{R}, L^{\frac{2N}{N-2}}(\mathbb{R}^N))} \rightarrow0,\ A \rightarrow \infty
$$
and by an interpolation inequality we can obtain \eqref{eq8.21} for every $H^1$-admissible pair $(s, r)$. In fact, consider a family of functions $\chi_R(t, x)=$ $\chi_R^1(t) \cdot \chi_R^2(x)$ in $\mathcal{S}(\mathbb{R}^N\times\mathbb{R})$ satisfying the following properties:
\begin{equation}\label{eq8.29}
\left\{\begin{array}{l}
|\tilde{\chi}_R^1|+|\hat{\chi}_R^2| \leq 2, \\
\operatorname{Supp}(\hat{\chi}_R^2) \subset\{\frac{1}{2 R} \leq|\xi| \leq 2 R\}, \\
\hat{\chi}_R^2 \equiv 1  \text { for }  \frac{1}{R} \leq|\xi| \leq R, \\
\tilde{\chi}_R^1(-\frac{|\xi|^2}{2})=1  \text { on } \operatorname{Supp}(\hat{\chi}_R^2),
\end{array}\right.
\end{equation}
where $\wedge$ denotes the partial Fourier transform in $x$ and $\sim$ denotes the partial Fourier transform in $t$. Note that
$$
\|P_n^A\|_{L^{\infty}(L^{\frac{2N}{N-2}}(\mathbb{R}^N))} \leq \|\chi_R * P_n^A \|_{L^{\infty}(\mathbb{R}, L^{\frac{2N}{N-2}}(\mathbb{R}^N))}+\|(\delta-\chi_R) * P_n^A\|_{L^{\infty}(\mathbb{R}, L^{\frac{2N}{N-2}}(\mathbb{R}^N))},
$$
where $*$ denotes the convolution in $(x, t)$ and $\delta$ denotes the Dirac distribution.

Now, we show that
\begin{equation}\label{eq8.30}
  \limsup _{n \rightarrow \infty}\|\chi_R * P_n^A\|_{L^{\infty}(\mathbb{R}, L^{\frac{2N}{N-2}}(\mathbb{R}^N))} \leq C(R) \eta(\mathbf{P}^A)^{\frac{2}{N}} \limsup _{n \rightarrow \infty} \|\nabla P_n^A\|_2^{1-\frac{2}{N}} .
\end{equation}
Indeed, by an interpolation inequality, we know
\begin{equation}\label{eq8.31}
\|\chi_R * P_n^A\|_{L^{\infty}(\mathbb{R}, L^{\frac{2N}{N-2}}(\mathbb{R}^N))} \leq  \|\chi_R * P_n^A \|_{L^{\infty}(\mathbb{R}, L^{2}(\mathbb{R}^N))}^{1-\frac{2}{N}}\|\chi_R * P_n^A\|_{L^{\infty}(\mathbb{R}, L^{\infty}(\mathbb{R}^N))}^{\frac{2}{N}} .
\end{equation}
Since $\chi_R * P_n^A$ is solution to \eqref{eq8.5} and the $L^2$-conservation law, it follows that
\begin{equation}\label{eq8.32}
\begin{aligned}
\|\chi_R * P_n^A\|_{L^{\infty}(\mathbb{R}, L^{2}(\mathbb{R}^N))}^2 & =\|(\chi_R * P_n^A)( \cdot,0)\|_{L_x^2(\mathbb{R}^N)}^2 \\
& =\frac{1}{(2 \pi)^N}\|\mathcal{F}_{x \rightarrow \xi}(\chi_R * P_n^A(\cdot,0))(\xi)\|_{L_{\xi}^2(\mathbb{R}^N)}^2 .
\end{aligned}
\end{equation}
Now, we write
$$
\chi_R * P_n^A(x, 0)=\int_{\mathbb{R}_s} \chi_R^1(-s) \int_{\mathbb{R}_y^N} \chi_R^2(x-y) P_n^A(y, s) d y d s.
$$
Using Plancherel inversion formula, it holds
$$
\chi_R * P_n^A(x, 0)=\frac{1}{(2 \pi)^N} \int_{\mathbb{R}_s} \chi_R^1(-s) \int_{\mathbb{R}_{\xi}^N} \hat{\chi}_R^2(-\xi) \widehat{P_n^A(\cdot, s)}(-\xi) e^{-i x \xi} d \xi d s.
$$
Note that $\widehat{P_n^A(\cdot, s)}(-\xi)=e^{-\frac{is|\xi|^2}{2}} \widehat{P_n^A(\cdot, s)}(-\xi)$, so
$$
\begin{aligned}
\chi_R * P_n^A(0, x) & =\frac{1}{(2 \pi)^N} \int_{\mathbb{R}_{\xi}^N} \tilde{\chi}_R^1\left(-\frac{|\xi|^2}{2}\right) \hat{\chi}_R^2(\xi) \widehat{P_n^A(\cdot, 0)}(\xi) e^{i x \xi} d \xi d s \\
& =\mathcal{F}_{\xi \rightarrow x}^{-1}\left[\tilde{\chi}_R^1\left(-\frac{|\xi|^2}{2}\right) \hat{\chi}_R^2(\xi) \widehat{P_n^A(\cdot,s)}(\xi)\right](x).
\end{aligned}
$$
Therefore,
\begin{equation}\label{eq8.33}
\mathcal{F}_{x \rightarrow \xi}(\chi_R * P_n^A(\cdot, 0))(\xi)=\tilde{\chi}_R^1(-\frac{|\xi|^2}{2}) \hat{\chi}_R^2(\xi) \widehat{P_n^A(\cdot, s)}(\xi) .
\end{equation}
Using the properties of $\chi_R$ listed in \eqref{eq8.29}, \eqref{eq8.32} and \eqref{eq8.33} we get
\begin{equation}\label{eq8.34}
\|\chi_R * P_n^A\|_{L^{\infty}(\mathbb{R}, L^{2}(\mathbb{R}^N))} \leq  C_1(R)\|\xi \widehat{P_n^A(\cdot, s)}\|_{L^2} \leqslant C_1(R) \|\nabla P_n^A\|_2,
\end{equation}
where $C_1(R)$ is an $R$-dependent constant. Now, observe that
$$
\limsup _{n \rightarrow \infty}\|\chi_R * P_n^A\|_{L^{\infty}(\mathbb{R}, L^{\infty}(\mathbb{R}^N))}=\sup _{(y_n, s_n)} \limsup _{n \rightarrow \infty}|\chi_R * P_n^A( y_n, s_n)| .
$$
By the definition of $\mathcal{V}(\mathbf{P}^{A})$, we get
$$
\underset{n \rightarrow \infty}{\limsup }\|\chi_R * P_n^A\|_{L^{\infty}(\mathbb{R}, L^{\infty}(\mathbb{R}^N))} \leq \sup \left\{\left|\int_{\mathbb{R}} \int_{\mathbb{R}^N} \chi_R(-t,-x) V(t, x) d x d t\right|, V \in \mathcal{V}(\mathbf{P}^{A})\right\} .
$$
Hence, by H\"older's inequality, it follows that
$$
\limsup _{n \rightarrow \infty}\left\|\chi_R * P_n^A\right\|_{L^{\infty}(\mathbb{R}, L^{\infty}(\mathbb{R}^N))} \leq  C_2(R) \sup {\|V\|_{L^{\infty}(\mathbb{R}, L^{\frac{2N}{N-2}}(\mathbb{R}^N))}}, V \in \mathcal{V}(\mathbf{P}^A)\},
$$
where $C_2(R)$ depends only on $R$. Since
$$
\|V\|_{L^{\infty}(\mathbb{R}, L^{\frac{2N}{N-2}}(\mathbb{R}^N))} \leq  C \|\nabla V\|_2 \leq C \eta(\mathbf{P}^A)
$$
for every $V \in \mathcal{V}(\mathbf{P}^A)$, we have
\begin{equation}\label{eq8.35}
\limsup _{n \rightarrow \infty}\left\|\chi_R * P_n^A\right\|_{L^{\infty}(\mathbb{R}, L^{\infty}(\mathbb{R}^N))} \leq  C_2(R) \eta(\mathbf{P}^A)
\end{equation}
for every $A \geq  1$. Using \eqref{eq8.31}, \eqref{eq8.34} and \eqref{eq8.35}, we obtain \eqref{eq8.30}.

Next, we claim that
\begin{equation}\label{eq8.36}
\limsup _{n \rightarrow \infty}\|(\delta-\chi_R) * P_n^A\|_{L^{\infty}(\mathbb{R}, L^{\frac{2N}{N-2}}(\mathbb{R}^N))}^2 \\
 \leq C \limsup _{n \rightarrow \infty} \int_{\{|  \xi | \leq \frac{1}{ R}\} \cup\{|\xi| \geq R\}}|\xi|^2|\hat{P}_n(\xi, 0)|^2 d \xi .
\end{equation}
In fact, since $(\delta-\chi_R)*P^A$ is a solution to \eqref{eq8.5} and Proposition \ref{P2.2}, we have
$$
\|(\delta-\chi_R) * P_n^A\|_{L^{\infty}(\mathbb{R}, L^{\frac{2N}{N-2}}(\mathbb{R}^N))}^2 \leq  C \|\nabla[(\delta-\chi_R) * P_n^A]\|_2^2 .
$$
By Plancherel and \eqref{eq8.33}, it follows that
$$
\|\nabla[(\delta-\chi_R) * P_n^A]\|_2^2=\frac{1}{(2 \pi)^N} \int_{\mathbb{R}_{\xi}^N}|\xi|^2\left|\widehat{P_n^A(\cdot, s)}(\xi)\left[1-\tilde{\chi}_R^1\left(-\frac{|\xi|^2}{2}\right) \hat{\chi}_R^2(\xi)\right]\right|^2 d \xi .
$$
Note that, by \eqref{eq8.29}, the quantity $\left[1-\tilde{\chi}_R^1(-\frac{|\xi|^2}{2}) \hat{\chi}_R^2(\xi)\right]$ is equal to zero for $\frac{1}{R} \leq |\xi| \leq R$ and uniformly bounded by 3. Consequently,
\begin{equation*}
 \limsup _{n \rightarrow \infty}\|(\delta-\chi_R) * P_n^A\|_{L^{\infty}(\mathbb{R}, L^{\frac{2N}{N-2}}(\mathbb{R}^N))}^2 \\
  \leq  C \limsup _{n \rightarrow \infty} \int_{\{|\xi| \leq  \frac{1}{R}\} \cup\{|\xi| \geq R\}}|\xi|^2\left|\hat{P}_n^A(\xi,0)\right|^2 d \xi .
\end{equation*}
Therefore, using \eqref{eq8.28}, we get \eqref{eq8.36}.

From estimates \eqref{eq8.30} and \eqref{eq8.36}, we have
\begin{eqnarray*}
&&\underset{n \rightarrow \infty}{\lim \sup }\|P_n^A\|_{L^{\infty}(\mathbb{R}, L^{\frac{2N}{N-2}}(\mathbb{R}^N))}^2 \\
&\leq& C \limsup _{n \rightarrow \infty}\left(\int_{\{|\xi| \leq  \frac{1}{R}\} \cup\{|\xi| \geq R\}}|\xi|^2 |\hat{P}_n(0, \xi)|^2 d \xi+C(R) \eta(\mathbf{P}^A)^{\frac{4}{N}} \|\nabla P_n^A\|_2^{2-\frac{4}{N}}\right) .
\end{eqnarray*}
Let $A$ go to infinity, then $R$ go to infinity and using the fact that $\eta(\mathbf{P}^A) \rightarrow0$ as $A \rightarrow \infty$, that the family of sequences $(\nabla P_n^A(\cdot,0))$ are uniformly bounded in $L^2(\mathbb{R}^N)$ and $\mathbf{1}$-oscillatory, we obtain
$$
\underset{n \rightarrow \infty}{\limsup }\|P_n^A\|_{L^{\infty}(\mathbb{R}, L^{\frac{2N}{N-2}}(\mathbb{R}^N))} \rightarrow0,\ A \rightarrow \infty
$$
as claimed. This completes the proof of Proposition \ref{P8.3}.
\end{proof}
\textbf{Step 3.} Complete the proof of Lemma \ref{L8.2}. Let us come back to the decomposition \eqref{eq8.99}. We set $P_n^j(y, t)=p_n^j(y, t)$. Note that  the sequence $((\mathcal{L})^{\frac{s}{2}}  P_n^j(\cdot, 0))_{n \geq0}$ is bounded and $\mathbf{1}$-oscillatory. For every $j \geq 1$, Proposition \ref{P8.3} provides a family $(e^{-it\mathcal{L}}V^{(j, \alpha)})_{\alpha \geq1}$ of solutions to \eqref{eq8.5} and a family $(y_n^{(j, \alpha)}, t_n^{(j, \alpha)})_{\alpha \geq 1} \subset \mathbb{R}^N\times\mathbb{R}$ such that
\begin{equation}\label{eq8.38}
  P_n^j(x, t)=\sum_{\alpha=1}^{A_j} e^{-i(t-t_n^{(j, \alpha)})\mathcal{L}^n}V^{(j, \alpha)}(x-y_n^{(j, \alpha)}, t-t_n^{(j, \alpha)}  )+P_n^{(j, A_j)}(x, t),
\end{equation}
where \eqref{eq8.21} and \eqref{eq8.22} hold. In terms of $p_n^j$, the identity \eqref{eq8.38} becomes
\begin{equation}\label{eq8.39}
p_n^j=\sum_{\alpha=1}^{A_j}e^{-i(t-t_n^{(j, \alpha)})\mathcal{L}^n} V^{(j, \alpha)}(x-x_n^{(j, \alpha)}, t-t_n^{(j, \alpha)})+w_n^{(j, A_j)}(x, t),
\end{equation}
where
$$
x_n^{(j, \alpha)}= y_n^{(j, \alpha)}, \  w_n^{(j, A_j)}(x, t)=  P_n^{(j, A_j)}(x, t) .
$$
Summing \eqref{eq8.99} and \eqref{eq8.39}, we get
\begin{equation}\label{eq8.40}
  v_n^{\prime}(t, x)=   \sum_{j=1}^l\left(\sum_{\alpha=1}^{A_j}e^{-i(t-t_n^{(j, \alpha)})\mathcal{L}^n} V^{(j, \alpha)}(x-x_n^{(j, \alpha)}, t-t_n^{(j, \alpha)})+w_n^{(j, A_j)}(x, t)\right) +q_n^l(x, t) .
\end{equation}
Equation \eqref{eq8.40} can be rewritten as
$$
v_n^{\prime}(t, x)=\sum_{j=1}^l\left(\sum_{\alpha=1}^{A_j} e^{-i(t-t_n^{(j, \alpha)})\mathcal{L}^n}V^{(j, \alpha)}(x-x_n^{(j, \alpha)}, t-t_n^{(j, \alpha)})\right)+w_n^{(l, A_1, \ldots, A_l)},
$$
where
$$
w_n^{(l, A_1, \ldots, A_l)}(x, t)=\sum_{j=1}^l w_n^{(j, A_j)}(x, t)+q_n^l(x, t).
$$
Now, using the same proof in \cite{KSK2001}, we have completed the proof of this lemma.
\end{proof}
\begin{proof}[\bf Proof of Lemma \ref{L8.1} ]
\eqref{eq8.1} is a consequence of the proof of Corollary 1.9 in \cite{KSK2001}, here, we use the hypothesis $\|e^{i t \mathcal{L}} v_{0, n}\|_{L_{t,x}^{\frac{2(N+2)}{N-2}}} \geq \rho>0$. \eqref{eq8.4} follows from the orthogonality of $(x_{j, n} ; t_{j, n})$ as in the proof of \eqref{eq8.3}. The rest of the lemma is contained in the proof of Theorem 1.12 in \cite{KSK2001}.
\end{proof}
\section{Compactness of critical element}

Let us consider the statement:

$(SC)$ For all $u_0 \in H^1(\mathbb{R}^N)$, with
\begin{equation*}
   \int_{\mathbb{R}^N}\left|\nabla u_0\right|^2dx<\int_{\mathbb{R}^N}|\nabla W|^2dx\ \text{and}\ E\left(u_0\right)<E(W),
\end{equation*}
if $u$ is the corresponding solution to the \eqref{eq1.1}, with maximal interval of existence $I$, then $I=(-\infty,+\infty)$ and $\|u\|_{L_{(-\infty, +\infty)}^{\frac{2(N+2)}{N-2}}W^{1,{\frac{2N(N+2)}{N^2+4}}}}<\infty$.

We say that $(S C)(u_0)$ holds if for this particular $u_0$(such as, take $\left.u\right|_{t=0}=u_0)$, with
\begin{equation*}
   \int_{\mathbb{R}^N}\left|\nabla u_0\right|^2dx<\int_{\mathbb{R}^N}|\nabla W|^2dx\ \text{and}\ E\left(u_0\right)<E(W)
\end{equation*}
and $u$ the corresponding solution to the (CP), with maximal interval of existence $I$, we have $I=(-\infty,+\infty)$ and $\|u\|_{L_{(-\infty, +\infty)}^{\frac{2(N+2)}{N-2}}W^{1,{\frac{2N(N+2)}{N^2+4}}}}<\infty$.

By Lemma \ref{L3.4}, $(S C)\left(u_0\right)$ holds if $\|u_0\|_{H^1}$ small. Hence, in light of Corollary \ref{c5.2}, there exists $\eta_0>0$ such that if $u_0$ is as in $(\mathrm{SC})$ and $E\left(u_0\right)<\eta_0$, then $(S C)\left(u_0\right)$ holds. Moreover, for any $u_0$ as in (SC), $E\left(u_0\right) \geq 0$ because of Theorem \ref{t5.1}. Thus, there exists a number $E_C$, with $\eta_0 \leq E_C \leq E(W)$, such that, if $u_0$ is as in (SC) and $E\left(u_0\right)<E_C$, $(S C)\left(u_0\right)$ holds and $E_C$ is optimal with this property. For the rest of this section we will assume that $E_C<E(W)$. We now prove that there exits a critical element $u_{0, C}$ at the critical level of energy $E_C$ so that $(S C)\left(u_{0, C}\right)$ does not hold and from the minimality, this element has a compactness property up to the symmetry of this equation. This is in fact a general principle which follows from the concentration compactness ideas. More precisely,
\begin{lemma}\label{L9.1}
There exists $u_{0, C}$ in $H^1$, with
$$
E(u_{0, C})=E_C<E(W), \quad \int_{\mathbb{R}^N}\left|\nabla u_{0, C}\right|^2dx<\int_{\mathbb{R}^N}|\nabla W|^2dx
$$
such that, if $u_C$ is the solution of \eqref{eq1.1} with data $u_{0, C}$, and maximal interval of existence $I, 0 \in \stackrel{\circ}{I}$, then $\left\|u_C\right\|_{L_{I}^{\frac{2(N+2)}{N-2}}W^{1,{\frac{2N(N+2)}{N^2+4}}}}=+\infty$.
\end{lemma}

\begin{lemma}\label{L9.2}
Assume $u_C$ is as in Lemmas \ref{L9.1} and  $\left\|u_C\right\|_{L_{I_{+}}^{\frac{2(N+2)}{N-2}}W^{1,{\frac{2N(N+2)}{N^2+4}}}}=+\infty$, where $I_{+}=(0,+\infty) \cap I$. Then there exists $x(t) \in \mathbb{R}^N$, for $t \in I_{+}$, such that
$$
K=\left\{v(x, t): v(x, t)=u_C\left(x-x(t), t\right)\right\}
$$
has the property that $\overline{K}$ is compact in $H^1$. A corresponding conclusion is reached if
\begin{equation*}
  \left\|u_C\right\|_{L_{I_{-}}^{\frac{2(N+2)}{N-2}}W^{1,{\frac{2N(N+2)}{N^2+4}}}}=+\infty,
\end{equation*}
where $I_{-}=(-\infty, 0) \cap I$.
\end{lemma}

\begin{lemma}\label{L9.3}
Let $\left\{z_{0, n}\right\} \in H^1$, with
\begin{equation*}
  \int_{\mathbb{R}^N}\left|\nabla z_{0, n}\right|^2dx<\int_{\mathbb{R}^N}|\nabla W|^2dx \  \text{and}\  E\left(z_{0, n}\right) \rightarrow E_C
\end{equation*}
and with $\left\|e^{i t \Delta} z_{0, n}\right\|_{L_{(-\infty, +\infty)}^{\frac{2(N+2)}{N-2}}W^{1,{\frac{2N(N+2)}{N^2+4}}}} \geq \rho$, where $\rho$ as in Lemma \ref{L3.4}. Let $\left\{V_{0, j}\right\}$ be as in Lemma \ref{L8.1}. Assume that one of the two hypothesis
 \begin{equation}\label{eq9.1}
   \varliminf_{n \rightarrow \infty} E(V_1^l(-t_{1, n}))<E_C
 \end{equation}
or after passing to a subsequence, we have that, with $s_n=-t_{1, n}$, $E(V_1^l(s_n)) \rightarrow E_C$, and $s_n \rightarrow s_* \in[-\infty,+\infty]$, and if $U_1$ is the non-linear profile (see Definition \ref{D2.1}) associated to $(V_{0,1},\{s_n\})$ we have that the maximal interval of existence of $U_1$ is $I=(-\infty,+\infty)$ and $\left\|U_1\right\|_{L_{(-\infty, +\infty)}^{\frac{2(N+2)}{N-2}}W^{1,{\frac{2N(N+2)}{N^2+4}}}}<\infty$ and
 \begin{equation}\label{eq9.2}
\varliminf_{n \rightarrow \infty} E(V_1^l(-t_{1, n}))=E_C .
 \end{equation}
Then (after passing to a subsequence), for $n$ large, if $z_n$ is the solution of \eqref{eq1.1} with data at $t=0$ equal to $z_{0, n}$, then $(S C)\left(z_{0, n}\right)$ holds.
\end{lemma}

Let us first assume the validity of Lemma \ref{L9.3} and use it (together with Lemma \ref{L8.1}) to establish Lemmas \ref{L9.1} and \ref{L9.2}.

\begin{proof}[\bf Proof of Lemma \ref{L9.1}]
According to the definition of $E_C$, and the assumption that $E_C<E(W)$, we can find $u_{0, n} \in H^1$, with
\begin{equation*}
  \int_{\mathbb{R}^N}\left|\nabla u_{0, n}\right|^2dx<\int_{\mathbb{R}^N}|\nabla W|^2dx,\ E\left(u_{0, n}\right) \rightarrow E_C
\end{equation*}
and such that if $u_n$ is the solution of \eqref{eq1.1} with data at $t=0$, $u_{0, n}$ and maximal interval of existence $I_n=\left(-T_{-}\left(u_{0, n}\right), T_{+}\left(u_{0, n}\right)\right)$, then $\left\|e^{i t \Delta} u_{0, n}\right\|_{L_{(-\infty, +\infty)}^{\frac{2(N+2)}{N-2}}W^{1,{\frac{2N(N+2)}{N^2+4}}}} \geq \rho>0$, where $\rho$ is as in Lemma \ref{L3.4} and $\left\|u_n\right\|_{L_{I_n}^{\frac{2(N+2)}{N-2}}W^{1,{\frac{2N(N+2)}{N^2+4}}}}=+\infty$(Here we are also using Proposition \ref{P2.1} and Lemma \ref{L3.4}). Note that, since $E_C<E(W)$, there exists $\delta_0>0$ such that
\begin{equation*}
  E\left(u_{0, n}\right) \leq\left(1-\delta_0\right) E(W),\ \forall n.
\end{equation*}
Because of Theorem \ref{t5.1}, we can find $\bar{\delta}$ so that
\begin{equation*}
   \int_{\mathbb{R}^N}\left|\nabla u_n(t)\right|^2dx \leq(1-\bar{\delta}) \int_{\mathbb{R}^N}|\nabla W|^2dx \ \text{for all}\ t \in I_n, \ \forall n.
\end{equation*}
Apply now Lemma \ref{L8.1} for $\varepsilon_0>0$ and Lemma \ref{L9.3}. We then have, for $J=J\left(\varepsilon_0\right)$, that
\begin{equation}\label{eq9.3}
  u_{0, n}   =\sum_{j=1}^J   V_j^l\left(x-x_{j, n},-t_{j, n}\right)+w_n,
\end{equation}
\begin{equation}\label{eq9.4}
 \|(\mathcal{L})^{\frac{s}{2}} v_{0, n}\|_2^2 =\sum_{j=1}^J \|(\mathcal{L})^{\frac{s}{2}}e^{it_{j, n}\mathcal{L}^{j,n}} V_{0, j}\|_2^2 + \|(\mathcal{L})^{\frac{s}{2}} w_n\|^2+o(1),\ n \rightarrow \infty,\ s\in\{0,1\},
\end{equation}
\begin{equation}\label{eq9.5}
    E\left(v_{0, n}\right)=\sum_{j=1}^J E (V_j^l ( -t_{j, n}  ) )+E\left(w_n\right)+o(1),\ n \rightarrow \infty.
\end{equation}
Note that because of \eqref{eq9.4} we have, for all $n$ large, that
\begin{equation*}
  \int_{\mathbb{R}^N}\left|\nabla w_n\right|^2dx \leq (1-\frac{\bar{\delta}}{2}) \int_{\mathbb{R}^N}|\nabla W|^2dx \ \text{and}\  \int_{\mathbb{R}^N}\left|\nabla V_{0, j}\right|^2dx \leq(1-\frac{\bar{\delta}}{2}) \int_{\mathbb{R}^N}|\nabla W|^2dx.
\end{equation*}
From Corollary \ref{c5.1} it now follows that $E (V_j^l (-t_{j, n}) ) \geq 0$ and $E\left(w_n\right) \geq 0$. From this and \eqref{eq9.5} it follows that
\begin{equation*}
  E(V_1^l(-t_{1, n})) \leq E\left(u_{0, n}\right)+o(1)
\end{equation*}
and hence $\varliminf\limits_{n \rightarrow \infty} E(V_1^l(-t_{1, n})) \leq E_C$. If the left-hand side is strictly less than $E_C$, Lemma \ref{L9.3} gives us a contradiction with the choice of $u_{0, n}$, for $n$ large (after passing to a subsequence). Hence, the left-hand side must equal $E_C$.

Let then $U_1$ be the non-linear profile associated to $(V_1^l,\{s_n\})$, with $s_n=-t_{1, n}$(after passing to a subsequence). We first note that we must have $J=1$. This is because \eqref{eq9.5} and $E(u_{0, n}) \rightarrow E_C, E(V_1^l(-s_n)) \rightarrow E_C$ now imply that
\begin{equation*}
  E(w_n) \rightarrow 0\ \text{and}\ E(V_j^l(-t_{j, n})) \rightarrow 0, j=2, \ldots, J.
\end{equation*}
Using \eqref{eq7.2} and the argument in the proof of Corollary \ref{c5.2}, we have
\begin{equation*}
\sum_{j=2}^J \|(\mathcal{L})^{\frac{s}{2}}V_j^l ( -t_{j, n} )\|_2^2 + \|(\mathcal{L})^{\frac{s}{2}} w_n\|^2 \rightarrow 0.
\end{equation*}
Since $\|(\mathcal{L})^{\frac{s}{2}}V_j^l ( -t_{j, n} )\|_2^2= \|(\mathcal{L})^{\frac{s}{2}}e^{it_{j, n}\mathcal{L}^{j,n}} V_{0, j}\|_2^2$, then we have
\begin{equation*}
  V_{0, j}=0, j=2, \ldots, J \text{ and } \|(\mathcal{L})^{\frac{s}{2}} w_n\|^2 \rightarrow 0.
\end{equation*}
Hence \eqref{eq9.3} becomes $u_{0, n}= V_1^l(x-x_{1, n}, s_n)+w_n$. Let $v_{0, n}=u_{0, n}(x+x_{1, n})$ and note that scaling gives us that $v_{0, n}$ verifies the same hypothesis as $u_{0, n}$. Moreover, $\widetilde{w}_n=w_n(x+x_{1, n})$ still verifies $\|(\mathcal{L})^{\frac{s}{2}} \widetilde{w}_n\|^2 \rightarrow 0$. Thus
$$
v_{0, n}=V_1^l(s_n)+\widetilde{w}_n, \ \|(\mathcal{L})^{\frac{s}{2}} \widetilde{w}_n\|^2 \rightarrow 0 .
$$

Let us return to $U_1$, the non-linear profile associated to $\left(V_{0,1},\left\{s_n\right\}\right)$ and let
\begin{equation*}
  I_1=(T_{-}(U_1), T_{+}(U_1))
\end{equation*}
be its maximal interval of existence. Note that, by definition of non-linear profile and Lemma \ref{L3.3}, we have
\begin{equation*}
  \int_{\mathbb{R}^N}\left|\nabla U_1\left(s_n\right)\right|^2dx=\int_{\mathbb{R}^N}\left|\nabla V_1^l\left(s_n\right)\right|^2dx+o(1)\ \text{and}\ E\left(U_1\left(s_n\right)\right)=E\left(V_1^l\left(s_n\right)\right)+o(1)
\end{equation*}
Note that in this case $E(V_1^l(s_n))=E_C+o(1)$ and
\begin{equation*}
  \int_{\mathbb{R}^N}\left|\nabla V_1^l\left(s_n\right)\right|^2dx= \int_{\mathbb{R}^N}\left|\nabla V_{0,1}\right|^2dx=\int_{\mathbb{R}^N}\left|\nabla u_{0, n}\right|^2dx+o(1)<\int_{\mathbb{R}^N}|\nabla W|^2dx
\end{equation*}
for $n$ large by Theorem \ref{t5.1}. Let's fix $\bar{s} \in I_1$. Then $E\left(U_1\left(s_n\right)\right)=E\left(U_1(\bar{s})\right)$, so that
$$
E\left(U_1(\bar{s})\right)=E_C .
$$
Moreover, $\int_{\mathbb{R}^N}\left|\nabla U_1\left(s_n\right)\right|^2dx<\int_{\mathbb{R}^N}|\nabla W|^2dx$ for $n$ large and hence by \eqref{eq7.4}
\begin{equation*}
  \int_{\mathbb{R}^N}\left|\nabla U_1(\bar{s})\right|^2dx<\int_{\mathbb{R}^N}|\nabla W|^2dx.
\end{equation*}
If $\left\|U_1\right\|_{L_{I_1}^{\frac{2(N+2)}{N-2}}W^{1,{\frac{2N(N+2)}{N^2+4}}}}<+\infty$, Remark \ref{r1.2} gives us that $I_1=(-\infty,+\infty)$ and we then obtain a contradiction from Lemma \ref{L9.3}. Thus,
$$
\left\|U_1\right\|_{L_{I_1}^{\frac{2(N+2)}{N-2}}W^{1,{\frac{2N(N+2)}{N^2+4}}}}=+\infty
$$
and we then set $u_C=U_1$$($after a translation in time to make $\bar{s}=0)$.
\end{proof}
\begin{proof}[\bf Proof of Lemma \ref{L9.2}]
By contradiction, let us set $u(x, t)=u_C(x, t)$ for convenience. If not, there exists $\eta_0>0$ and a sequence $\left\{t_n\right\}_{n=1}^{\infty}$, $t_n \geq 0$ such that, for all $x_0 \in \mathbb{R}^N$, we have
\begin{equation}\label{eq9.6}
  \left\|u(x-x_0, t_n)-u(x, t_{n^{\prime}})\right\|_{H^1} \geq \eta_0,\ \text { for } n \neq n^{\prime} .
\end{equation}
Note that$($after passing to a subsequence, so that $t_n \rightarrow \bar{t} \in\left[0, T_{+}\left(u_0\right)\right])$, we must have $\bar{t}=T_{+}\left(u_0\right)$, in view of the continuity of the flow in $H^1$, as guaranteed by Lemma \ref{L3.4}. Note that, in view of Lemma \ref{L3.4} we must also have $\left\|e^{i t \mathcal{L}} u(t_n)\right\|_{L_{(0, +\infty)}^{\frac{2(N+2)}{N-2}}W^{1,{\frac{2N(N+2)}{N^2+4}}}} \geq \rho$.

\textbf{Step 1.} Let us apply Lemma \ref{L8.1} to $v_{0, n}=u\left(t_n\right)$ with $\varepsilon_0>0$. We will show that $J=1$. Indeed, if $\varliminf\limits_{n \rightarrow \infty} E(V_1^l(-t_{1, n}))<E_C$, then by Theorem \ref{t5.1}, we have
\begin{equation*}
  \int_{\mathbb{R}^N}|\nabla u(t)|^2dx \leq (1-\delta_1) \int_{\mathbb{R}^N}|\nabla W|^2dx \ \text{for all}\  t \in I_{+}
\end{equation*}
and $E(u(t))=E\left(u_0\right)=$ $E_C<E(W)$, by Lemma \ref{L9.3} we obtain that $(SC)(u)$ holds. So $\left\|u\right\|_{L_{I_{+}}^{\frac{2(N+2)}{N-2}}W^{1,{\frac{2N(N+2)}{N^2+4}}}}<+\infty$, which contradicts the hypothesis. Therefore, it follows that $\varliminf\limits_{n \rightarrow \infty} E(V_1^l(-t_{1, n}))=E_C$. Similar to the proof of Lemma \ref{L9.1}, we get \begin{equation*}
  J=1,\ \|(\mathcal{L})^{\frac{s}{2}} w_n\|^2 \rightarrow 0.
\end{equation*}
Thus, we have
\begin{equation}\label{eq9.7}
  u(t_n)= V_1^l(x-x_{1, n},-t_{1, n})+w_n, \quad \|(\mathcal{L})^{\frac{s}{2}} w_n\|^2 \rightarrow 0 .
\end{equation}

\textbf{Step 2.} We prove that $s_n=-t_{1, n}$ must be bounded. In fact, note that
$$
e^{i t \mathcal{L}} u(t_n)= V_1^l(x-x_{1, n}, t-t_{1, n})+e^{i t \mathcal{L}} w_n .
$$

On the other hand, assume $t_{1, n}\leq-C_0$, where $C_0$ is a large positive constant. Then, since
\begin{equation*}
  \left\|e^{i t \mathcal{L}} w_n\right\|_{L_{(0, +\infty)}^{\frac{2(N+2)}{N-2}}W^{1,{\frac{2N(N+2)}{N^2+4}}}}<\frac{\rho}{ 2} \ \text{for}\ n \ \text{large enough}
\end{equation*}
and
$$
\|V_1^l(x-x_{1, n}, t-t_{1, n})\|_{L_{(0, +\infty)}^{\frac {2(N+2)}{N-2}}W^{1,{\frac{2N(N+2)}{N^2+4}}}} \leq\left\|V_1^l(y, s)\right\|_{L_{(C_0, +\infty)}^{\frac{2(N+2)}{N-2}}W^{1,{\frac{2N(N+2)}{N^2+4}}}} \leq \frac{\rho}{2}
$$
for $C_0$ large, which contradicts  $\left\|e^{i t \mathcal{L}} u\left(t_n\right)\right\|_{L_{(0, +\infty)}^{\frac{2(N+2)}{N-2}}W^{1,{\frac{2N(N+2)}{N^2+4}}}} \geq \rho$.

On the other hand, assume that $t_{1, n}\geq C_0$, for a large positive constant $C_0, n$ large, we have
\begin{equation*}
  \|V_1^l(x-x_{1, n}, t-t_{1, n})\|_{L_{(-\infty,0)}^{\frac{2(N+2)}{N-2}}W^{1,{\frac{2N(N+2)}{N^2+4}}}}   \leq\left\|V_1^l(y, s)\right\|_{L_{(-\infty,-C_0)}^{\frac{2(N+2)}{N-2}}W^{1,{\frac{2N(N+2)}{N^2+4}}}} \leq \frac{\rho}{2}
\end{equation*}
for $C_0$ large. Hence, $\left\|e^{i t \mathcal{L}} u(t_n)\right\|_{L_{(-\infty,0)}^{\frac{2(N+2)}{N-2}}W^{1,{\frac{2N(N+2)}{N^2+4}}}} \leq \rho$, for $n$ large. By Lemma \ref{L3.4}, we know that $\|u\|_{L_{(-\infty,t_n)}^{\frac{2(N+2)}{N-2}}W^{1,{\frac{2N(N+2)}{N^2+4}}}} \leq \rho$, which gives us a contradiction because of $t_n \rightarrow T_{+}(u_0)$. Thus $\left|t_{1, n}\right| \leq C_0$ and after passing to a subsequence,
$$
t_{1, n}\rightarrow t_0 \in(-\infty,+\infty) .
$$

\textbf{Step 3.} By \eqref{eq9.6} and \eqref{eq9.7}, for $n \neq n^{\prime}$ large$($independently of $x_0)$, it holds
\begin{equation*}
  \left\|V_1^l(x-x_0-x_{1, n},-t_{1, n})-V_1^l(x-x_{1, n^{\prime}},-t_{1, n^{\prime}}) \right\|_{H^1} \geq \frac{\eta_0}{2}
\end{equation*}
or
\begin{equation*}
  \left\| V_1^l(y+\widetilde{x}_{n, n^{\prime}}-\widetilde{x}_0,-t_{1, n})
  -V_1^l(y,-t_{1, n^{\prime}}) \right\|_{H^1} \geq \frac{\eta_0}{2},
\end{equation*}
where $\widetilde{x}_{n, n^{\prime}}$ is a suitable point in $\mathbb{R}^N$ and $\widetilde{x}_0$ are arbitrary. But if we choose $\tilde{x}_0 =x_{n, n^{\prime}}$, then $-t_{1, n} \rightarrow -t_0$ and $-t_{1, n^{\prime}} \rightarrow-t_0$. So $  \left\|0\right\|_{H^1} \geq \frac{\eta_0}{2}$, which reaches a contradiction.
\end{proof}

Thus, to complete the proofs of Lemmas \ref{L9.1} and \ref{L9.2} we only need to provide the proof of Lemma \ref{L9.3}.

\begin{proof}[\bf Proof of Lemma \ref{L9.3}]
Let us assume first that \eqref{eq9.2} holds and set
\begin{equation*}
  A= \int_{\mathbb{R}^N}|\nabla W|^2dx, A^{\prime}=\int_{\mathbb{R}^N}|\nabla W|^2dx, M=\left\|U_1\right\|_{L_{(-\infty, +\infty)}^{\frac{2(N+2)}{N-2}}W^{1,{\frac{2N(N+2)}{N^2+4}}}}.
\end{equation*}
Arguing (for some $\varepsilon_0>0$ in Lemma \ref{L8.1}) as in the proof of Lemmas \ref{L9.1}, we see that
\begin{equation*}
  \varliminf\limits_{n \rightarrow \infty} E(V_1^l(-t_{1, n})) =E_C \text{ and } E_C<E(W),
\end{equation*}
which imply that $J=1$, $\|(\mathcal{L})^{\frac{s}{2}} w_n\|^2 \rightarrow 0$. Moreover, if
 \begin{equation*}
v_{0, n} = z_{0, n}(x+x_{1, n}),\ \widetilde{w}_n=  w_n (x+x_{1, n}),\ s_n=- t_{1, n} ,
 \end{equation*}
we have $\|(\mathcal{L})^{\frac{s}{2}} \widetilde{w}_n\|^2  \rightarrow 0$ and $v_{0, n}=V_1^l\left(s_n\right)+\widetilde{w}_n$, while
\begin{equation*}
  \left\|e^{i t \Delta} v_{0, n}\right\|_{L_{(-\infty, +\infty)}^{\frac{2(N+2)}{N-2}}W^{1,{\frac{2N(N+2)}{N^2+4}}}} \geq \delta,\ \int_{\mathbb{R}^N}\left|\nabla v_{0, n}\right|^2dx<\int_{\mathbb{R}^N}|\nabla W|^2dx,\ E\left(v_{0, n}\right) \rightarrow E_C.
\end{equation*}
By definition of non-linear profile, we know that
\begin{equation*}
  \int_{\mathbb{R}^N}|\nabla  V_1^l(s_n)- \nabla U_1(s_n)|^2dx=o(1).
\end{equation*}
We then have
$$
v_{0, n}=U_1(s_n)+\widetilde{\widetilde{w}}_n, \|(\mathcal{L})^{\frac{s}{2}} \widetilde{\widetilde{w}}_n\|^2  \rightarrow 0 .
$$
Moreover, as in the proof of Lemma \ref{L9.1}, $E(U_1(0))=E_C$ and $\int_{\mathbb{R}^N}\left|\nabla U_1(t)\right|^2dx<\int_{\mathbb{R}^N}|\nabla W|^2dx$ for all $t$. We now apply Proposition \ref{P5.1}, with $\varepsilon_0<\varepsilon_0\left(M, A, A^{\prime}, N\right)$ and $n$ large, with $\tilde{u}=U_1, e \equiv 0, t_0=0, u_0=v_{0, n}$. This case now follows.

Next, assume that \eqref{eq9.1} holds, the proof is divided into five steps.

\textbf{Step 1 } We prove that for $j \geq 2$, we also have $\varliminf\limits_{n \rightarrow \infty} E(V_j^l(-t_{j, n}))<E_C$. In fact, up to a subsequence, assume $\lim\limits_{n \rightarrow \infty} E(V_1^l(-t_{1, n}))<E_C$. Due to \eqref{eq8.3}, it holds
$$
\int_{\mathbb{R}^N} |(\mathcal{L})^{\frac{s}{2}} z_{0, n} |^2dx \geq \sum\limits_{j=1}^J \int_{\mathbb{R}^N} |(\mathcal{L})^{\frac{s}{2}} V_{0, j} |^2dx+o(1)
$$
and since $E_C<E(W)$, for $n$ large we have $E(z_{0, n}) \leq(1-\delta_0) E(W)$, by Lemma \ref{L7.1},
\begin{equation*}
  \int_{\mathbb{R}^N}\left|\nabla z_{0, n}\right|^2dx \leq(1- \delta_1) \int_{\mathbb{R}^N}|\nabla W|^2dx \text{ and } \int_{\mathbb{R}^N}|\nabla V_{0, j}|^2dx \leq (1-\delta_1) \int_{\mathbb{R}^N}|\nabla W|^2dx.
\end{equation*}
Similarly, $\int_{\mathbb{R}^N}|\nabla w_n|^2dx \leq(1-\delta_1) \int_{\mathbb{R}^N}|\nabla W|^2dx$. By Corollary \ref{c5.1}, we have
\begin{equation*}
  E(V_j^l(-t_{j, n})) \geq 0,\ E(w_n) \geq 0.
\end{equation*}
Moreover, using \eqref{eq8.1} and the proof of Corollary \ref{c5.2}, we have
\begin{equation*}
  E(V_1^l(-t_{1, n})) \geq C \int_{\mathbb{R}^N}\left|\nabla V_{0,1}\right|^2dx \geq c \alpha_0=\overline{\alpha_0}>0 \text{ for } n  \text{ large}.
\end{equation*}
By \eqref{eq8.4}, it holds
$$
E(z_{0, n}) \geq \overline{\alpha_0}+\sum_{j=2}^J E(V_j^l(-t_{j, n}))+o(1) \text{ for } n  \text{ large},
$$
so the claim follows from $E\left(z_{0, n}\right) \rightarrow E_C$.

\textbf{Step 2 } We show that (after passing to a subsequence so that, for each $j$, $\lim\limits_n E(V_j^l(-t_{j, n}))$ exists and $\lim\limits_n(-t_{j, n})=\overline{s_j} \in[-\infty,+\infty]$ exists) if $U_j$ is the non-linear profile associated to $(V_j^l,\{-t_{j, n}\})$, then $U_j$ satisfies (SC). Indeed, according to the definition of non-linear profile and Step 1, it follows that $E(U_j)<E_C$ because of $\varliminf\limits_{n \rightarrow \infty} E(V_j^l(-t_{j, n}))<E_C$. Moreover, since
\begin{equation*}
  \int_{\mathbb{R}^N}|\nabla V_j^l(-t_{j, n})|^2dx \leq (1-\delta_1)  \int_{\mathbb{R}^N}|\nabla W|^2dx,
\end{equation*}
the definition of non-linear profile and Theorem \ref{t5.1}, if $\bar{t} \in I_j($the maximal interval for $U_j)$, we have $\int_{\mathbb{R}^N}|\nabla U_j(\bar{t})|^2dx<\int_{\mathbb{R}^N}|\nabla W|^2dx$. By the definition of $E_C$, our claim follows. Note that the argument in the proof of Proposition \ref{P5.1} also gives that $\left\| U_j\right\|_{L_{(-\infty,+\infty)}^\frac{2(N+2)}{N-2}W^{1,{\frac{2N(N+2)}{N^2+4}}}}<+\infty$.

\textbf{Step 3 } We claim that there exists $j_0$ so that, for $j \geq j_0$ we have
 \begin{equation}\label{eq9.8}
\left\|U_j\right\|_{L_{(-\infty,+\infty)}^\frac{2(N+2)}{N-2}W^{1,{\frac{2N(N+2)}{N^2+4}}}}^{\frac{2(N+2)} {N-2}} \leq C\| V_{0, j}\|_{H^1}^{\frac{2(N+2)}{ N-2}} .
 \end{equation}
In fact, from \eqref{eq8.3}, for fixed $J$ we see that (choosing $n$ large)
 \begin{equation*}
   \sum_{j=1}^J \int_{\mathbb{R}^N}\left|\nabla V_{0, j}\right|^2dx \leq \int_{\mathbb{R}^N}\left|\nabla z_{0, n}\right|^2dx+o(1) \leq 2 \int_{\mathbb{R}^N}|\nabla W|^2dx.
 \end{equation*}
Thus, for $j \geq j_0$, we have
\begin{equation*}
  \int_{\mathbb{R}^N}\left|\nabla V_{0, j}\right|^2dx \leq \widetilde{\delta},
\end{equation*}
where $\tilde{\delta}$ is so small that $\left\|e^{i t \Delta} V_{0, j}\right\|_{L_{(-\infty,+\infty)}^\frac{2(N+2)}{N-2}W^{1,{\frac{2N(N+2)}{N^2+4}}}} \leq \rho$, with $\rho$ as in Lemma \ref{L3.1}. From the definition of non-linear profile, it then follows that $\left\|U_j\right\|_{L_{(-\infty,+\infty)}^\frac{2(N+2)}{N-2}W^{1,{\frac{2N(N+2)}{N^2+4}}}} \leq 2 \rho$, and using the integral equation
\begin{equation*}
  u(t)=e^{-i t \mathcal{L}} u_0+i\int_0^t e^{-i\left(t-t^{\prime}\right) \mathcal{L}} |u|^{\frac{4}{N-2}}u d t^{\prime}.
\end{equation*}
So $\left\|U_j(0)\right\|_{H^1} \leq C\left\|V_{0, j}\right\|_{H^1}$ and $\left\| U_j\right\|_{L_{(-\infty,+\infty)}^\frac{2(N+2)}{N-2}W^{1,{\frac{2N(N+2)}{N^2+4}}}} \leq C\left\|V_{0, j}\right\|_{H^1}$, which implies \eqref{eq9.8}.

\textbf{Step 4 } For $\varepsilon_0>0$, to be chosen, define now
$$
H_{n, \varepsilon_0}=\sum\limits_{j=1}^{J\left(\varepsilon_0\right)}  U_j\left(x-x_{j, n}, t-t_{j, n}\right),
$$
then it follows that
\begin{equation}\label{eq9.9}
  \left\|H_{n, \varepsilon_0}\right\|_{L_{(-\infty,+\infty)}^\frac{2(N+2)}{N-2}L^{\frac{2N(N+2)}{N-2}}} \leq C_0,
\end{equation}
uniformly in $\varepsilon_0$, for $n \geq n(\varepsilon_0)$. In fact,
\begin{eqnarray*}
&&\left\|H_{n, \varepsilon_0}\right\|_{L_{(-\infty,+\infty)}^\frac{2(N+2)}{N-2}L^{\frac{2N(N+2)}{N-2}}}\\
&= & \iint\left[\sum_{j=1}^{J\left(\varepsilon_0\right)} U_j\left(x-x_{j, n}, t-t_{j, n}\right)\right]^{\frac{2(N+2)}{N-2}} \\
&\leq & C_{J\left(\varepsilon_0\right)} \sum_{j^{\prime} \neq j} \iint\left| U_j\left(x-x_{j, n}, t-t_{j, n}\right)\right|\cdot\left| U_{j^{\prime}}\left(x-x_{j^{\prime}, n}, t-t_{j^{\prime}, n}\right)\right|^{\frac{N+6}{N-2}}\\
&& +\sum_{j=1}^{J\left(\varepsilon_0\right)} \iint\left| U_j\left(x-x_{j, n}, t-t_{j, n}\right)\right|^{\frac{2(N+2)}{N-2}} \\
&=&\mathrm{I}+\mathrm{II} .
\end{eqnarray*}
By the orthogonality of $(\lambda_{j, n} ; x_{j, n} ; t_{j, n})$, we know that $\mathrm{II} \rightarrow 0$ for $n$ large(see Keraani \cite{KSK2001}). Hence, for $n$ large we have $\mathrm{II} \leq \mathrm{I}$. Since \eqref{eq8.3}, it follows that
\begin{eqnarray*}
\mathrm{I} & \leq& \sum_{j=1}^{j_0}\left\|U_j\right\|_{L_{(-\infty,+\infty)}^\frac{2(N+2)}{N-2}L^{\frac{2N(N+2)}{N-2}}}^{\frac{2(N+2)}{ N-2}}+\sum_{j=j_0}^{J\left(\varepsilon_0\right)}\left\|U_j\right\|_{L_{(-\infty,+\infty)}^\frac{2(N+2)}{N-2}L^{\frac{2N(N+2)}{N-2}}} ^{\frac{2(N+2)}{ N-2}}\\
& \leq& \sum_{j=1}^{j_0}\left\|U_j\right\|_{L_{(-\infty,+\infty)}^\frac{2(N+2)}{N-2}L^{\frac{2N(N+2)}{N-2}}}^{\frac{2(N+2)}{ N-2}}+C \sum_{j=j_0}^{J\left(\varepsilon_0\right)}\| V_{0, j}\|_{H^1}^{\frac{2(N+2)}{ N-2}}  \\ &\leq& \frac{C_0}{2},
\end{eqnarray*}
where $j_0$ is defined as in \eqref{eq9.8}. For $\varepsilon_0>0$, to be chosen, define
\begin{equation*}
  R_{n, \varepsilon_0}=\left|H_{n, \varepsilon_0}\right|^{\frac{4}{N-2}} H_{n, \varepsilon_0}-  \sum_{j=1}^{J\left(\varepsilon_0\right)}\left| U_j\left(x-x_{j, n}, t-t_{j, n}\right)\right|^{\frac{4}{N-2}} U_j\left(x-x_{j, n}, t-t_{j, n}\right).
\end{equation*}
using the arguments of Keraani  \cite{KSK2001}, we get
$$
\text { For } n=n\left(\varepsilon_0\right) \text { large, }\left\|\nabla R_{n, \varepsilon_0}\right\|_{L_t^2 L_x^{\frac{2 N}{N+2}}} \rightarrow 0 \quad \text { as } n \rightarrow \infty \text {. }
$$

\textbf{Step 5 }  Finally, we apply Proposition \ref{P5.1} to obtain our purpose. Let
\begin{equation*}
  \widetilde{u}=H_{n, \varepsilon_0},\ e=R_{n, \varepsilon_0},
\end{equation*}
where $\varepsilon_0$ is still to be determined. Recall that
\begin{equation*}
  z_{0, n}=\sum_{j=1}^{J\left(\varepsilon_0\right)} V_j^l\left(x-x_{j, n}, -t_{j, n}\right) +w_n,
\end{equation*}
where $\left\|e^{i t \mathcal{L}} w_n\right\|_{L_{(-\infty,+\infty)}^\frac{2(N+2)}{N-2}L^{\frac{2N(N+2)}{N-2}}} \leq \varepsilon_0$. By the definition of non-linear profile, we now have
$$
z_{0, n}(x)=H_{n, \varepsilon_0}(x, 0)+\widetilde{w}_n(x),
$$
where, for $n$ large $\left\|e^{i t \mathcal{L}} \widetilde{w}_n\right\|_{L_{(-\infty,+\infty)}^\frac{2(N+2)}{N-2}L^{\frac{2N(N+2)}{N-2}}} \leq 2 \varepsilon_0$. Moreover, according to the orthogonality of $\left(\lambda_{j, n} ; x_{j, n} ; t_{j, n}\right)$ and Corollary \ref{c5.2}, for $n=n\left(\varepsilon_0\right)$ large, it holds
\begin{equation*}
  \int_{\mathbb{R}^N}\left|\nabla H_{n, \varepsilon_0}(t)\right|^2dx \leq 2 \sum_{j=1}^{J\left(\varepsilon_0\right)} \int_{\mathbb{R}^N}|\nabla U_j(t-t_{j, n})|^2dx \leq 4 C \sum_{j=1}^{J\left(\varepsilon_0\right)} \int_{\mathbb{R}^N}\left|\nabla V_{0, j}\right|^2dx
\end{equation*}
and
\begin{equation*}
  \sum_{j=1}^{J\left(\varepsilon_0\right)} \int_{\mathbb{R}^N}\left|\nabla V_{0, j}\right|^2dx \leq \int_{\mathbb{R}^N}\left|\nabla z_{0, n}\right|^2dx+\int_{\mathbb{R}^N}\left|\nabla z_{0, n}\right|^2dx+o(1) \leq 2 \int_{\mathbb{R}^N}|\nabla W|^2dx.
\end{equation*}
Let $M=C_0$ with $C_0$ as in \eqref{eq9.8},
\begin{equation*}
  A=\widetilde{C} \int_{\mathbb{R}^N}|\nabla W|^2dx, A^{\prime}=A+\int_{\mathbb{R}^N}|\nabla W|^2dx, \varepsilon_0<\frac{\varepsilon_0(M, A, A^{\prime}, N)}{2},
\end{equation*}
where $\varepsilon_0(M, A, A^{\prime}, N)$ is defined as in Proposition \ref{P5.1}. Fix $\varepsilon_0$ and choose $n$ so large that $\|\nabla R_{n, \varepsilon_0}\|_{L_T^\infty H^1\cap L_T^2 L_x^{\frac{2 N}{N+2}}}<\varepsilon_0$ and so that all the above properties hold. Then Proposition \ref{P5.1} indicates that the conclusion is valid in the case when \eqref{eq8.1} holds.
\end{proof}
\begin{remark}\label{r8.1}
Assume that $\left\{z_{0, n}\right\}$ in Lemma \ref{L8.1} and $V(x)$ are all radial. Then $V_{0, j}, w_n$ can be chosen to be radial and we can choose $x_{j, n} \equiv 0$. This follows directly from Keraani's proof \cite{KSK2001}. If we then define $(\mathrm{SC})$ and $E_C$ by restricting only to radial functions, we obtain a $u_C$ as in Lemma \ref{L9.1} which is radial, and we can establish Lemma \ref{L9.2} with $x(t) \equiv 0$.
\end{remark}
\section{Proof of Theorem \ref{t1.4}}
Now, we will eliminate a minimal blow-up solution.
\begin{lemma}\label{L10.1}
Assume that $u_0 \in H^1$ is such that
$$
E(u_0)<E(W), \ \int_{\mathbb{R}^N}\left|\nabla u_0\right|^2dx<\int_{\mathbb{R}^N}|\nabla W|^2dx.
$$
Assume that $u$ be the solution of \eqref{eq1.1} and $\left.u\right|_{t=0}=u_0$ with maximal interval of existence $(-\infty, +\infty)$. If
$$
K=\left\{u\left(x, t\right): t \in[0, +\infty)\right\}
$$
is such that $\overline{K}$ is compact in $H^1$. Then $u_0 \equiv 0$.
\end{lemma}
\begin{remark}\label{r10.1}
We conjecture that Theorem \ref{L10.1} remains true if $v(x, t)= u(x-x(t), t)$, with $x(t) \in \mathbb{R}^N, t \in[0, +\infty)$. In other words, for ``energy subcritical'' initial data, compactness up to the invariances of the equation, for solutions, is only true for $u \equiv 0$.
\end{remark}

In the next lemma we will collect some useful facts:
\begin{lemma}\label{L10.2}
Let $u, v$ be as in Remark \ref{L10.1}.

i) Let $\delta_0>0$ be such that $E(u_0) \leq(1-\delta_0) E(W)$. Then for all $t \in$ $\left[0, T_{+}\left(u_0\right)\right)$, we have
$$
\begin{gathered}
\int_{\mathbb{R}^N}|\nabla u(t)|^2dx \leq(1-\delta_1) \int_{\mathbb{R}^N}|\nabla W|^2dx, \\
\int_{\mathbb{R}^N}(|\nabla u|^2-|u|^{2^*})dx \geq \bar{\delta} \int_{\mathbb{R}^N}|\nabla u|^2dx,\\
C_{1, \delta_0} \int_{\mathbb{R}^N}\left|\nabla u_0\right|^2dx \leq E(u_0) \leq C_2 \int_{\mathbb{R}^N}\left|\nabla u_0\right|^2dx, \\
E(u(t))=E\left(u_0\right), \\
C_{1, \delta_0} \int_{\mathbb{R}^N}\left|\nabla u_0\right|^2dx \leq \int_{\mathbb{R}^N}|\nabla u(t)|^2dx \leq C_2 \int_{\mathbb{R}^N}\left|\nabla u_0\right|^2dx .
\end{gathered}
$$

ii)
$$
\begin{aligned}
& \int_{\mathbb{R}^N}|\nabla v(t)|^2dx \leq C_2 \int_{\mathbb{R}^N}|\nabla W|^2dx, \\
& \|v(t)\|_{L_x^{2^*}}^2 \leq C_3 \int_{\mathbb{R}^N}|\nabla W|^2dx .
\end{aligned}
$$

iii) For all $x_0 \in \mathbb{R}^N$
$$
\int_{\mathbb{R}^N} \frac{|v(x, t)|^2}{\left|x-x_0\right|^2}dx \leq C_4 \int_{\mathbb{R}^N}|\nabla W|^2 dx.
$$

iv) For each $\varepsilon_0>0$, there exists $R(\varepsilon_0)>0$, such that, for $0 \leq t<T_{+}(u_0)$, we have
$$
\int_{|x| \geq R(\varepsilon_0)}\left(|\nabla v|^2dx+|v|^{2^*}+\frac{|v|^2}{|x|^2}\right)dx \leq \varepsilon_0.
$$
\end{lemma}
\begin{proof}
Using Theorem \ref{t5.1}, Corollary \ref{c5.2} and Sobolev embedding, it is easy to see that i) and ii) hold. iii) follows from Hardy's inequality. using Sobolev embedding and the Hardy inequality, follows from the compactness of $ \overline{K}$.
\end{proof}
The next lemma is a localized virial identity about potential equation. The proof idea comes from Merle \cite{FM1992}.

\begin{lemma}\label{L10.3}
Let $\psi \in C_0^{\infty}(\mathbb{R}^N), t \in\left[0, T_{+}\left(u_0\right)\right)$. Then:

i)
$$
\frac{d}{d t} \int_{\mathbb{R}^N}|u|^2 \psi d x=2 \operatorname{Im} \int_{\mathbb{R}^N} \bar{u} \nabla u \nabla \psi d x
$$

ii)
\begin{eqnarray*}
\frac{d^2}{d t^2} \int_{\mathbb{R}^N}|u|^2 \psi d x&=&4\int_{\mathbb{R}^N}\Delta \psi |\nabla u|^2 d x-2\int_{\mathbb{R}^N} \nabla \psi \nabla V(x) |u|^2dx-\int_{\mathbb{R}^N}\Delta^2 \psi | u|^2 d x\\
&&-\frac{4}{N}\int_{\mathbb{R}^N} \Delta \psi |u|^{\frac{2N}{N-2}}  dx.
\end{eqnarray*}
\end{lemma}
\begin{proof}
By \eqref{eq1.1} and direct calculation, we get
\begin{eqnarray*}
\frac{d}{d t} \int_{\mathbb{R}^N} \psi(x)|u(t, x)|^2 d x & =&2 \int_{\mathbb{R}^N} \operatorname{Re} \psi \frac{\partial u}{\partial t} \bar{u}dx \\
& =&2 \int_{\mathbb{R}^N} \operatorname{Re}\left[i\Delta u-iV(x)u +i|u|^{\frac{4}{N-2}}u\right] \bar{u} \psi dx \\
& =&-2 \operatorname{Im} \int_{\mathbb{R}^N} \Delta u \bar{u} \psi dx\\
&=&2 \operatorname{Im} \int_{\mathbb{R}^N} \nabla u \bar{u} \nabla \psi dx
\end{eqnarray*}
because of $\operatorname{Im} \int_{\mathbb{R}^N} \nabla u \overline{\nabla u}   \psi dx=0$ and part (i) follows.

Using (i), we know that
\begin{eqnarray*}
\frac{d^2}{d t^2} \int_{\mathbb{R}^N} \psi(x)|u(t, x)|^2 d x & =&2\left[\operatorname{Im} \int_{\mathbb{R}^N} \nabla \psi \bar{u} \nabla \frac{\partial u}{\partial t}dx+\operatorname{Im} \int_{\mathbb{R}^N} \nabla \psi \frac{\overline{\partial u}}{\partial t} \nabla udx\right] \\
& =&2\left[2 \operatorname{Im} \int_{\mathbb{R}^N} \nabla \psi \frac{\overline{\partial u}}{\partial t} \nabla udx-\operatorname{Im} \int_{\mathbb{R}^N} \Delta \psi \bar{u} \frac{\partial u}{\partial t}dx\right] .
\end{eqnarray*}
On the one hand,
\begin{eqnarray*}
&&-\operatorname{Im} \int_{\mathbb{R}^N} \Delta \psi \bar{u} \frac{\partial u}{\partial t}dx\\
& =&-\operatorname{Re} \int_{\mathbb{R}^N} \Delta \psi \bar{u}\left(\Delta u-V(x)u +|u|^{\frac{4}{N-2}}u\right)dx \\
& =&\int_{\mathbb{R}^N} \Delta \psi V(x)|u|^2dx-\int_{\mathbb{R}^N} \Delta \psi |u|^\frac{2N}{N-2}dx+\int_{\mathbb{R}^N} \Delta \psi |\nabla u|^2dx+\frac{1}{2} \int_{\mathbb{R}^N} \nabla|u|^2 \nabla(\Delta \psi)dx.
\end{eqnarray*}
On the other hand, note that
\begin{eqnarray*}
  -2 \operatorname{Re} \int_{\mathbb{R}^N} \nabla \psi \Delta u \overline{\nabla u}dx &=& 2 \operatorname{Re} \int_{\mathbb{R}^N} \Delta \psi \nabla u \overline{\nabla u}dx+2 \operatorname{Re} \int_{\mathbb{R}^N} \nabla \psi  \overline{\Delta u} \nabla u dx \\
    &=& 2 \operatorname{Re} \int_{\mathbb{R}^N} \Delta \psi \nabla u \overline{\nabla u}dx+2 \operatorname{Re} \int_{\mathbb{R}^N} \nabla \psi   \overline{\Delta u \overline{\nabla u}} dx   \\
    &=& 2 \operatorname{Re} \int_{\mathbb{R}^N} \Delta \psi \nabla u \overline{\nabla u}dx+2 \operatorname{Re} \int_{\mathbb{R}^N} \nabla \psi \Delta u \overline{\nabla u}dx,
\end{eqnarray*}
so
\begin{equation*}
   -2 \operatorname{Re} \int_{\mathbb{R}^N} \nabla \psi \Delta u \overline{\nabla u}dx=  \operatorname{Re} \int_{\mathbb{R}^N} \Delta \psi \nabla u \overline{\nabla u}dx.
\end{equation*}
Therefore,  we have
\begin{eqnarray*}
&&2 \operatorname{Im} \int_{\mathbb{R}^N} \nabla \psi \frac{\overline{\partial u}}{\partial t} \nabla u dx\\
&=&-2 \operatorname{Im} \int_{\mathbb{R}^N} \nabla \psi \frac{\partial u}{\partial t} \overline{\nabla u}dx\\
&= & -2 \operatorname{Re} \int_{\mathbb{R}^N} \nabla \psi( \Delta u-V(x)u +|u|^{\frac{4}{N-2}}u)\overline{\nabla u}dx \\
&= & -2 \operatorname{Re} \int_{\mathbb{R}^N} \nabla \psi \Delta u \overline{\nabla u}dx+2 \operatorname{Re} \int_{\mathbb{R}^N} \nabla \psi V(x)u \overline{\nabla u}dx-2 \operatorname{Re} \int_{\mathbb{R}^N} \nabla \psi (|u|^{\frac{4}{N-2}}u) \overline{\nabla u}dx \\
&= &\int_{\mathbb{R}^N} \Delta \psi |\nabla u|^2dx+ \int_{\mathbb{R}^N} \nabla \psi V(x)\nabla|u|^2dx+\frac{N-2}{N}\int_{\mathbb{R}^N} \nabla \psi \nabla(|u|^{\frac{2N}{N-2}}) dx\\
&= &\int_{\mathbb{R}^N} \Delta \psi |\nabla u|^2dx- \int_{\mathbb{R}^N} \Delta \psi V(x) |u|^2dx-\int_{\mathbb{R}^N} \nabla \psi \nabla V(x) |u|^2dx+\frac{N-2}{N}\int_{\mathbb{R}^N} \Delta \psi |u|^{\frac{2N}{N-2}}  dx.
\end{eqnarray*}
Combining the above two estimates, we obtain
 \begin{eqnarray*}
&&\frac{d^2}{d t^2} \int_{\mathbb{R}^N} \psi(x)|u(t, x)|^2 d x \\
& =&2\left[2 \operatorname{Im} \int_{\mathbb{R}^N} \nabla \psi \frac{\overline{\partial u}}{\partial t} \nabla udx-\operatorname{Im} \int_{\mathbb{R}^N} \Delta \psi \bar{u} \frac{\partial u}{\partial t}dx\right]\\
& =&2\left[\int_{\mathbb{R}^N} \Delta \psi |\nabla u|^2dx- \int_{\mathbb{R}^N} \Delta \psi V(x) |u|^2dx-\int_{\mathbb{R}^N} \nabla \psi \nabla V(x) |u|^2dx+\frac{N-2}{N}\int_{\mathbb{R}^N} \Delta \psi |u|^{\frac{2N}{N-2}}  dx\right.\\
&&\left.+\int_{\mathbb{R}^N} \Delta \psi V(x)|u|^2dx-\int_{\mathbb{R}^N} \Delta \psi |u|^\frac{2N}{N-2}dx+\int_{\mathbb{R}^N} \Delta \psi |\nabla u|^2dx+\frac{1}{2} \int_{\mathbb{R}^N} \nabla|u|^2 \nabla(\Delta \psi)dx\right]\\
& =&4\int_{\mathbb{R}^N}\Delta \psi |\nabla u|^2 d x-2\int_{\mathbb{R}^N} \nabla \psi \nabla V(x) |u|^2dx-\int_{\mathbb{R}^N}\Delta^2 \psi | u|^2 d x-\frac{4}{N}\int_{\mathbb{R}^N} \Delta \psi |u|^{\frac{2N}{N-2}}  dx,
\end{eqnarray*}
which implies that the conclusion is valid.
\end{proof}
\begin{proof}[\bf Proof of Lemma \ref{L10.1}]
By compactness in $H^1$, We first note that
for each $\varepsilon>0$, there exists $R(\varepsilon)>0$ such that, for all $t \in[0, \infty)$, it holds
\begin{equation}\label{eq10.3}
  \int_{|x|>R(\varepsilon)}\left(|\nabla u|^2+|u|^{\frac{2N}{N-2}} +\frac{|u|^2}{|x|^2}\right)dx \leq \varepsilon.
\end{equation}
Moreover, there exists $R_0>0$ such that, for all $t \in[0,+\infty)$,
\begin{equation}\label{eq10.5}
   8\int_{|x| \leq R_0}|\nabla u|^2dx- 8\int_{|x| \leq R_0}|u|^{\frac{2N}{N-2}}dx \geq C_{\delta_0} \int_{\mathbb{R}^N}\left|\nabla u_0\right|^2dx .
\end{equation}
In fact, \eqref{eq7.5} combined with Lemma \ref{L10.2} i) yields
\begin{equation*}
  8 \int_{\mathbb{R}^N}|\nabla u|^2dx-8 \int_{\mathbb{R}^N}|u|^{\frac{2N}{N-2}}dx \geq \widetilde{C}_{\delta_0} \int_{\mathbb{R}^N}\left|\nabla u_0\right|^2dx.
\end{equation*}
Now combine this with \eqref{eq10.3}, with $\varepsilon=\varepsilon_0 \int_{\mathbb{R}^N}\left|\nabla u_0\right|^2dx$ to obtain \eqref{eq10.5}.

To prove Case 2, we choose $\varphi \in C_0^{\infty}(\mathbb{R}^N)$, radial, with $\varphi(x)=|x|^2$ for $|x| \leq 1, \varphi(x) \equiv 0$ for $|x| \geq 2$. Define
\begin{equation*}
  z_R(t)=\int_{\mathbb{R}^N}|u(x, t)|^2 R^2 \varphi(\frac{x}{R}) d x,
\end{equation*}
then we have
\begin{eqnarray*}
 \left|z_R^{\prime}(t)\right| \leq C_{N, \delta_0} \int_{\mathbb{R}^N}\left|\nabla u_0\right|^2 R^2dx\text { for } t>0, \\
 z_R^{\prime \prime} \geq C_{N, \delta_0} \int_{\mathbb{R}^N}\left|\nabla u_0\right|^2dx\text { for } R \text { large enough, } t>0.
\end{eqnarray*}
In fact, from Lemma \ref{L10.3} i),
\begin{eqnarray*}
\left|z_R^{\prime}(t)\right| & \leq& 2 R\left|\operatorname{Im} \int_{\mathbb{R}^N}  \overline{ u} \nabla u \nabla \varphi(\frac{x}{R}) d x\right| \leq C_N R \int_{0 \leq x \mid \leq 2 R} \frac{|x|}{|x|}|\nabla u||u|dx \\
& \leq& C_N R^2\left(\int_{\mathbb{R}^N}|\nabla u|^2dx\right)^{\frac{1}{2}}\left(\int_{\mathbb{R}^N} \frac{|u|^2}{|x|^2}dx\right)^{\frac{1}{2}} \\
&\leq& C_N R^2 \int_{\mathbb{R}^N}\left|\nabla u_0\right|^2dx
\end{eqnarray*}
because of Lemma \ref{L10.2} i). In view of $x\cdot\nabla V\leq0$, H\"older inequality, Sobolev inequality, \eqref{eq10.3}, \eqref{eq10.5} and Lemma \ref{L10.3}, ii), we have
\begin{eqnarray*}
z_R^{\prime \prime}(t)&=& 4\int_{\mathbb{R}^N}\Delta \varphi |\nabla u|^2 d x-2\int_{\mathbb{R}^N} \nabla \varphi \nabla V(x) |u|^2dx-\int_{\mathbb{R}^N}\Delta^2 \varphi | u|^2 d x\\
&=& 8N\int_{|x|\leq R}  |\nabla u|^2 d x+4\int_{R\leq|x|\leq 2R}\Delta \varphi |\nabla u|^2 d x-4\int_{|x|\leq R}(x\cdot\nabla V) |u|^2dx\\
&&-2\int_{R\leq|x|\leq 2R} \nabla \varphi \nabla V(x) |u|^2dx-\int_{R\leq|x|\leq 2R}\Delta^2 \varphi | u|^2 d x-8\int_{|x|\leq R}   |u|^{\frac{2N}{N-2}}  dx\\
&&-\frac{4}{N}\int_{R\leq|x|\leq 2R} \Delta\varphi |u|^{\frac{2N}{N-2}}  dx\\
&=& 8N\int_{|x|\leq R}  |\nabla u|^2 d x+4\int_{R\leq|x|\leq 2R}\Delta \varphi |\nabla u|^2 d x-\frac{4}{N}\int_{R\leq|x|\leq 2R} \Delta\varphi |u|^{\frac{2N}{N-2}}  dx\\
&&-C\|\nabla V\|_{L^{\frac{N}{2}}}\|u\|_{2^*}^2 -\int_{R\leq|x|\leq 2R}\Delta^2 \varphi | u|^2 d x-8\int_{|x|\leq R}   |u|^{\frac{2N}{N-2}}  dx\\
& \geq&  8\int_{|x|\leq R}[  |\nabla u|^2-|u|^{\frac{2N}{N-2}} ]d x -C_N \int_{R \leq|x| \leq 2 R}\left[|\nabla u|^2+\frac{|u|^2}{|x|^2}+|u|^{2^*}\right]dx\\
&\geq& C_{N, \delta_0} \int_{\mathbb{R}^N}\left|\nabla u_0\right|^2dx
\end{eqnarray*}
for $R$ large. If we now integrate in $t$, we have $z_R^{\prime}(t)-z_R^{\prime}(0) \geq C_{N, \delta_0} t \int_{\mathbb{R}^N}\left|\nabla u_0\right|^2dx$, but we also have $\left|z_R^{\prime}(t)-z_R^{\prime}(0)\right| \leq 2 C_N R^2 \int_{\mathbb{R}^N}\left|\nabla u_0\right|^2dx$, a contradiction for $t$ large, unless $\int_{\mathbb{R}^N}\left|\nabla u_0\right|^2dx=0$.
\end{proof}

\begin{proof}[\bf Proof of Theorem \ref{t1.4}]
By the integral equation in Lemma \ref{L3.4}, we know that $u(t)$ is radial for each $t \in I$. Using Remark \ref{r8.1} and Lemma \ref{L10.1} we obtain $$
I=(-\infty,+\infty),\|u\|_{L_I^\frac{2(N+2)}{N-2}W^{1,{\frac{2N(N+2)}{N^2+4}}}}<+\infty.
$$
 Now Remark \ref{r4.1} finishes the proof.
\end{proof}

\end{document}